\documentclass[reqno,12pt]{amsart}
\usepackage{amsmath,amssymb}

\newtheorem{theorem}{Theorem}[section]
\newtheorem{proposition}[theorem]{Proposition}
\newtheorem{definition}[theorem]{Definition}
\newtheorem{lemma}[theorem]{Lemma}
\newtheorem{corollary}[theorem]{Corollary}
\newtheorem{remark}[theorem]{Remark}

\numberwithin{equation}{section}
\numberwithin{theorem}{section}

\newcommand{\mc}[1]{{\mathcal #1}}
\newcommand{\mf}[1]{{\mathfrak #1}}
\newcommand{\mb}[1]{{\mathbf #1}}
\newcommand{\bb}[1]{{\mathbb #1}}
\newcommand{\eps}{\varepsilon}
\newcommand{\upbar}[1]{\,\overline{\! #1}}
\newcommand{\downbar}[1]{\underline{#1\!}\,}
\newcommand{\Glimsup}{\mathop{\textrm{$\Gamma\!\!$--$\varlimsup$}}}
\newcommand{\Gliminf}{\mathop{\textrm{$\Gamma\!\!$--$\varliminf$}}}
\newcommand{\Glim}{\mathop{\textrm{$\Gamma\!\!$--$\lim$}}}
\newcommand{\un}[1]{{\mathbf 1}_{#1}}

\title{A $\Gamma$-convergence approach to Large deviations}

\author[M.\ Mariani]{Mauro Mariani} 
\address{Mauro Mariani, 
Faculty of Mathematics, National Research University Higher School of Economics,
6 Usacheva St.,
119048 Moscow, Russia}
\email{mmariani@hse.ru}

\begin{document}
\begin{abstract}
A rigorous connection between large deviations theory and $\Gamma$-convergence is established. Applications include representations formulas for rate functions, a contraction principle for measurable maps, a large deviations principle for coupled systems and a second order Sanov theorem.
\end{abstract}

\maketitle

\section{Introduction}
\label{s:1}

Let $X$ be a Polish space, that is a completely metrizable, separable topological space. The space $\mc P(X)$ of Borel probability measures on $X$ is a Polish space as well, if equipped with the so-called \emph{narrow} (otherwise called \emph{weak}) topology. Such a topology enjoys several characterizations, see \cite[Theorem~3.1.5]{St}. A sequence $(\mu_n)$ in $\mc P(X)$ converges narrowly to $\mu$ iff $\varlimsup_n \mu_n(C)\le \mu(C)$ or $\varliminf_n \mu_n(O)\ge \mu(O)$ for all $C\subset X$ closed and $O\subset X$ open,  or equivalently iff the integrals of bounded continuous functions converge.

A \emph{Large Deviations principle} (LDP) for $(\mu_n)$ on $X$ is then classically defined as an exponential version of the inequalities on closed and open sets stated above for the narrow convergence; and the Brycs-Varadhan theorem \cite[Chapter 4.4]{DZ} can be regarded as a Large Deviations' (LD) analog of the characterization of narrow convergence by the convergence of integrals of continuous bounded functions.

In this paper, we further extend the analogies between narrow convergence and LD to other characterizations. At least when $X$ is compact, it is easy to see that the narrow convergence $\mu_n \to \mu$ is equivalent to the $\Gamma$-convergence of the relative entropy functional $H(\cdot|\mu_n)$ to $H(\cdot|\mu)$ and also to the $\Gamma$-convergence of the maps $K\mapsto -\log \mu_n(K)$ to $K\mapsto -\log \mu(K)$, where the compact subsets $K$ of $X$ are equipped with the Hausdorff topology (which indeed coincides with the Kuratowski topology on compact sets). In section~\ref{s:3} we provide the LD analogs of these statements, proving in particular that LD is also a notion of convergence in a metric space $\mc W(X)$, containing both probability measures and functionals. Thus convergence of measures to measures in $\mc W(X)$ is equivalent to the narrow convergence, convergence of functionals to functionals is equivalent to $\Gamma$-convergence, and convergence of measures to functionals is indeed LD, see Theorem~\ref{t:topol} for a precise statement. It is worth to remark that various approaches to LD are possible by the means of variational analysis of the relative entropy functional \cite{DE}, the one in this paper being indeed inspired by the techniques in \cite{J,Ma}.

In section~\ref{s:4} we apply the results in section~\ref{s:3} to get some general properties of LDP. Proposition~\ref{p:exrepr} gives some explicit representations of the LD rate functionals, that generalize the so-called Laplace-Varadhan method for proving LDPs. In Proposition~\ref{p:contr} a version of a so-called contraction principle is provided for measurable (not just continuous) contraction maps. In Theorem~\ref{t:coupled} we give sufficient conditions to recover a LDP for a coupled system of metric random variables, from the LDPs for the (independent) components of an associated system with frozen variables. In Theorem~\ref{t:sanov}, we apply the results in section~\ref{s:3}-\ref{s:4} to provide a second order version of the Sanov theorem for triangular arrays of i.i.d.\ random variables whose law also satisfies a LDP (see the discussion in section~\ref{s:6} for applications).

\section{Preliminaries}
\label{s:2}
In this section we recall the basic notions concerning $\Gamma$-convergence and LD.
Hereafter $\mf B(X)$ denotes be the Borel $\sigma$-algebra on the Polish space
$X$ and $\mc P(X)$ the set of Borel probability measures on $X$. For $\mu \in \mc P(X)$ and $f$ a $\mu$-integrable function on $X$, $\mu(f)$ will denote the integral of $f$ with respect to $\mu$. $\mc P(X)$ is hereafter equipped with the narrow topology, namely the weakest topology such that the maps $\mathcal{P}(X) \ni \mu \mapsto \mu(f) \in \bb R$ are continuous, for all $f \in C_\mathrm{b}(X)$.

We also let $\mf K(X)$ be the collection of compact subsets of $X$, equipped with the Hausdorff topology\footnote{If $X$ is compact, the Hausdorff topology on $\mf K(X)$ coincides with the Kuratowski topology, see \cite[Chapter~4]{RW}. The latter is often considered in the theory of $\Gamma$-convergence, due its equivalence with the Kuratowski convergence of epigraphs, see e.g.\ \cite[Theorem~4.16]{DM}. Here we use a slightly different construction, as we lift functions and measures on $X$ to functionals on $\mf K(X)$. The Hausdorff topology comes more handy, since $\mf K(X)$ is itself Polish. The price to pay is that possibly some of the statements in the paper would extend to closed sets, not just compact sets, if the Kuratowski topology would be used. However, as long as one sticks with exponentially tight families of probabilities (or uniformly coercive functions) the two notions of convergence are equivalent, so that the price to pay for working with the Hausdorff topology is negligible for all the applications discussed here.
}. 
Namely, fixed a compatible distance $d$ on $X$, define $d_H\colon \mf K(X) \times \mf K(X) \to [0,+\infty[$ as
\begin{equation*}
d_H(K,K'):=\inf \{\eps>0,\,K\subset {K'}^\eps,\,K'\subset K^\eps\}
\end{equation*}
where, for $A\in \mf B(X)$, $A^\eps$ denotes the $\eps$-enlargement of $A$ with respect to the distance $d$. As well known, $d_H$ defines a distance on $\mf K(X)$, and the associated topology $\tau_H$ does not depend on the choice of the compatible distance $d$. Moreover, $(\mf K(X),\tau_H)$ is a Polish space, see \cite[Chapter~4]{RW}, and it is understood that $\mf K(X)$ is equipped with such a topology in the following.

\subsection{$\Gamma$-convergence}
\label{s:2.1}
$\Gamma$-convergence is the relevant notion of convergence for functionals, whenever problems related to minima and minimizers are investigated.
\begin{definition}
\label{d:Igood}
A functional $I\colon X\to [0,+\infty]$ is \emph{lower semicontinuous} iff
for each $\ell \ge 0$ the set $\{x\in X\,:\: I(x)\le \ell\}$ is closed. $I$ is \emph{coercive} iff
for each $\ell\ge 0$ the set $\{x\in X\,:\: I(x)\le \ell\}$ is precompact.


Let $(I_n)$ be a sequence of functionals $I_n\colon X\to [0,+\infty]$. $(I_n)$ is 
\emph{equicoercive} on $X$ iff for each $\ell>0$,
$\cup_n \{x \in X \,:\: I_n(x) \le \ell \}$ is precompact.
\end{definition}

\begin{definition}
\label{d:Glim}
  The $\Gamma$-liminf (also denoted $\Gliminf$) and $\Gamma$-limsup
  (also denoted $\Glimsup$) of a sequence $(I_n)$ of functionals $I_n \colon X \to [0,+\infty]$ are two functionals on $X$ defined as follows. For $x \in X$
  \begin{equation*}
    \begin{split}
\big( \Gliminf_n I_n \big) (x):=
     \inf \big\{ \varliminf_n I_n(x_n),\,
\text{
     $(x_n)$ sequence in $X$ such that  $x_n \to x$} \big\}
      \\
 \big( \Glimsup_n I_n \big) (x):=
     \inf \big\{ \varlimsup_n I_n(x_n),\,\text{
     $(x_n)$ sequence in $X$ such that  $x_n \to x$} \big\}
    \end{split}
  \end{equation*}
  Whenever $\Gliminf I_n =\Glimsup I_n=I$, $(I_n)$ is said to
  $\Gamma$-converge to $I$ in $X$, and $I$ is called the
  $\Gamma$-limit (also denoted $\Glim$) of $(I_n)$.
\end{definition}

\subsection{Large deviations}
\label{s:2.2}
Hereafter  $(\mu_n)$ is a sequence in $\mathcal{P}(X)$  and $(a_n)$ is a sequence of positive reals such that $\lim_n a_n=+\infty$.

\begin{definition}
\label{d:exptight}
The sequence $(\mu_n)$ is \emph{exponentially tight} with speed $(a_n)$ iff 
\begin{equation*}
  \label{e:exptight}
  \inf_{K \subset X,\,\text{compact}} \varlimsup_n \tfrac{1}{a_n} \log
  \mu_n(K^c)=-\infty
\end{equation*}
\end{definition}

\begin{definition}
\label{d:ld}
Let $I\colon X \to [0,+\infty]$ be a lower semicontinuous functional. Then $(\mu_n)$ satisfies
\begin{itemize}
 \item {A \emph{LD lower bound} with speed $(a_n)$
and rate $I$, iff for each open set $O \subset X$
\begin{equation*}
 \varliminf_n \tfrac{1}{a_n}  \log \mu_n(O) \ge - \inf_{x \in O}
I(x)
\end{equation*}
}

\item {A \emph{LD weak upper bound} with speed $(a_n)$
and rate $I$, iff for each compact $K \subset X$
\begin{equation*}
 \varlimsup_n \tfrac{1}{a_n} \log \mu_n(K) \le - \inf_{x \in K}
I(x)
\end{equation*}
}

\item {A \emph{LD upper bound} with speed $(a_n)$
and rate $I$, iff for each closed set $C \subset X$
\begin{equation*}
 \varlimsup_n \tfrac{1}{a_n} \log \mu_n(C) \le - \inf_{x \in C}
I(x)
\end{equation*}
}
\end{itemize}

$(\mu_n)$ satisfies a \emph{(weak) LDP} if both the lower and (weak) upper bounds hold with same rate
and speed.

\end{definition}
It is immediate to check that if  $(\mu_n)$ is exponentially tight and satisfies a weak LD upper bound, then it satisfies a LD upper bound.

\subsection{Relative entropy}
\label{ss:2.3}
Given $\mu,\,\nu \in \mathcal{P}(X)$ and $\mf F \subset \mf B(X)$ a
$\sigma$-algebra, the relative entropy of $\nu$ with respect to $\mu$ on $\mf F$ is defined as
\begin{equation}
\label{e:H1}
H_{\mf F}(\nu|\mu):=
 \sup_{\varphi} \big\{\nu(\varphi) - \log \mu(e^{\varphi}) \big\}
\end{equation}
where the supremum runs over the bounded $\mf F$-measurable
functions $\varphi$ on $X$. For a fixed $\mu$, $H_{\mf F}(\cdot|\mu)$ is a positive, convex functional on $\mathcal{P}(X)$. If $\mf F=\mf B(X)$, the subindex $\mf F$ will be dropped hereafter. In such a case, $H(\cdot|\mu)$ is also lower semicontinuous and coercive on $\mc P(X)$.

\subsection{Regular set-maps}

If $a>0$, $\mu\in \mathcal{P}(X)$ and $I\colon X\to [0,+\infty]$ a lower semicontinuous functional, define the set-maps $l_{a,\mu},\,l_I \colon  \mf K(X) \to [0,+\infty]$ as
\begin{equation}
\label{e:lamu}
l_{a,\mu}(K)= -\tfrac{1}{a} \log \mu(K)
\end{equation}
\begin{equation}
\label{e:lI}
l_{I}(K)=\inf_{x\in K} I(x)
\end{equation}
Since probability measures are regular on Polish spaces, it is easy to check that $l_{a,\mu}$ is lower semicontinuous on $\mf K(X)$, while the lower semicontinuity of $I$ implies that $l_I$ is lower semicontinuous as well.

\section{Large Deviations and $\Gamma$-convergence}
\label{s:3}
 The equivalence of probabilistic statements concerning LD (labeled P), $\Gamma$-convergence statements concerning relative entropies (labeled H) and set-maps (labeled L) is established in this section.

An equivalent formulation of narrow convergence of probability measures is first introduced in section~\ref{ss:3.1}. Although only needed in proofs to appear later in the paper, it gives an easy example of the ideas concerning the analogous LD statements in section~\ref{ss:3.2}. Proofs are provided in section~\ref{ss:3.3}.

\subsection{Weak convergence and relative entropy}
\label{ss:3.1}
Let $(\mu_n)$ be a sequence in $ \mathcal{P}(X)$, and define $H_n \colon \mathcal{P}(X) \to
[0,+\infty]$ as
\begin{equation*}
H_n(\nu):= H(\nu|\mu_n)
\end{equation*}
The parameter $a>0$ has no special role in the next two propositions, one could fix $a=1$. Yet, it will become relevant when LD are considered.

\begin{proposition}
\label{p:tight}
The following are equivalent.
\begin{itemize}
\item[(P)]{$(\mu_n)$ is tight in $\mathcal{P}(X)$.}
\item[(H)]{$(H_n)$ is equicoercive on $\mc P(X)$.}
\item[(L)]{For $a>0$,  $(l_{a,\mu_n})$ is equicoercive on $\mf K(X)$.}
\end{itemize}
\end{proposition}

\begin{proposition}
\label{p:conv}
The following are equivalent.
\begin{itemize}
\item[(P1)]{$\mu_n \to \mu$ in $\mathcal{P}(X)$.}

\item[(P2)]{For each sequence $(\varphi_n)$ of Borel measurable
    functions $\varphi_n\colon X \to \bar{\bb R}$ bounded from below
\begin{equation*}
\varliminf_n \mu_n (\varphi_n) 
       \ge \mu(\Gliminf_n \varphi_n)
\end{equation*}
}
 
\item[(H)]{$ (\Glim_n H_n)(\nu)=H(\nu|\mu)$. }

\item[(L1)]{For $a>0$, $\Glim_n l_{a,\mu_n}=l_{a,\mu}$.}

\item[(L2)]{For $a>0$, $\Glimsup_n l_{a,\mu_n} \le l_{a,\mu}$.}

\item[(L3)]{For $a>0$, $\Gliminf_n l_{a,\mu_n} \ge l_{a,\mu}$.}

\end{itemize}
\end{proposition}

\subsection{Large Deviations and relative entropy}
\label{ss:3.2}
In this section the LD analogs of
Proposition~\ref{p:tight} and Proposition~\ref{p:conv} are
stated. Hereafter $\mb a=(a_n)$ is a sequence of strictly positive real
numbers such that $\lim_n a_n=+\infty$, $(\mu_n)$ is a sequence in
$\mathcal{P}(X)$ and $I \colon X \to [0,+\infty]$ a measurable function. Define
$H^{\mb a}_n\colon \mathcal{P}(X) \to [0,+\infty]$ as
\begin{equation*}
H_n^{\mb a}(\nu):=\tfrac{1}{a_n} H(\nu|\mu_n)
\end{equation*}
and, recalling \eqref{e:lamu}, $l^{\mb a}_n \colon \mf K(X) \to [0,+\infty]$ as
\begin{equation*}
l^{\mb a}_n=l_{a_n,\mu_n}
\end{equation*}

\begin{theorem}
\label{t:tightld}
The following are equivalent.
\begin{itemize}
\item[(P)]{$(\mu_n)$ is exponentially tight with speed $(a_n)$.}
\item[(H)]{$(H^{\mb a}_n)$ is equicoercive.}
\item[(L)]{$(l^{\mb a}_n)$ is equicoercive.}
\end{itemize}
\end{theorem}

\begin{theorem}
\label{t:convldlow}
The following are equivalent.
\begin{itemize}
\item[(P1)]{$(\mu_n)$ satisfies a LD lower
    bound with speed $(a_n)$ and rate $I$.}

\item[(P2)]{For each sequence $(\varphi_n)$ of measurable
    maps $\varphi_n\colon X \to \overline{\bb R}$
\begin{equation*}
  \varliminf_n \tfrac 1{a_n} \log \mu_n \big(\exp(a_n \varphi_n) \big) 
  \ge \sup_{x \in X} 
        \big\{\big(\Gliminf_n \varphi_n \big)(x) - I(x)\big\}
\end{equation*}
where one understands $\big(\Gliminf_n \varphi_n \big)(x) - I(x)=-\infty$
whenever $I(x)=+\infty$. }

\item[(H1)]{For each $x \in X$, $\big(\Glimsup_n
    H^{\mb a}_n\big)(\delta_x) \le I(x)$, where $\delta_x \in \mathcal{P}(X)$ is
    the Dirac mass concentrated at $x$.}

\item[(H2)]{For each $\nu \in \mathcal{P}(X)$, $\big(\Glimsup_n
    H^{\mb a}_n\big)(\nu) \le \nu(I)$.}
\end{itemize}

If $I$ is lower semicontinuous, the above statements are also equivalent to
\begin{itemize}   
\item[(L)]{ $\Glimsup_n l^{\mb a}_n \le l_I$.}
\end{itemize}
\end{theorem}

\begin{theorem}
\label{t:convldup}
Assume that $I$ is lower semicontinuous. Then the following are equivalent.
\begin{itemize}
\item[(P1)]{$(\mu_n)$ satisfies a LD weak upper
    bound with speed $(a_n)$ and rate $I$.}

\item[(P2)]{For each sequence $(\varphi_n)$ of measurable maps
    $\varphi_n\colon X \to \overline{\bb R}$ bounded from below and such that 
\begin{equation}
\label{e:kkkii}
\sup_{K \subset X\,\text{compact}} \varlimsup_n
    \frac{\mu_n\big(\un {K^c} \exp(-a_n \varphi_n )\big)}{\mu_n\big( \exp(-a_n \varphi_n )\big)} =0
\end{equation}    
the following inequality holds
\begin{equation*}
\varlimsup_n \tfrac 1{a_n} \log \mu_n\big(\exp(-a_n \varphi_n) \big) 
 \le \sup_{x \in X} \big\{ -\big(\Gliminf_n \varphi_n \big)(x)- I(x)\big\}
\end{equation*}
where $-\big(\Gliminf_n \varphi_n \big)(x)- I(x)\colon =-\infty$
whenever $I(x)=+\infty$.  }

\item[(H1)]{For each $x \in X$, $\big(\Gliminf_n
    H^{\mb a}_n\big)(\delta_x) \ge I(x)$.}
    
\item[(H2)]{For each $\nu \in \mathcal{P}(X)$, $\big(\Gliminf_n
    H^{\mb a}_n\big)(\nu) \ge \nu(I)$.}
    
\item[(L)]{ $\Gliminf_n l^{\mb a}_n \ge l_I$.}
\end{itemize}

Assume furthermore that $(\mu_n)$ satisfies the equivalent
conditions of Theorem~\ref{t:tightld}. Then the above statements are also equivalent to
\begin{itemize}
\item[(P3)]{$\mu_n$ satisfies a LD upper
    bound with speed $(a_n)$ and rate $I$.}

\item[(P4)]{For each sequence $(\varphi_n)$ of measurable maps
    $\varphi_n\colon X \to \overline{\bb R}$, bounded from below it holds
\begin{equation*}
\varlimsup_n \tfrac 1{a_n} \log \bb \mu_n \big(\exp(-a_n \varphi_n) \big) 
 \le \sup_{x \in X} \big\{ -\big(\Gliminf_n \varphi_n \big)(x)- I(x)\big\}
\end{equation*}
where $-\big(\Gliminf_n \varphi_n \big)(x)- I(x):=-\infty$ whenever
$I(x)=+\infty$.}
\end{itemize}
\end{theorem}

\subsection{Proofs for section~\ref{s:3}}
\label{ss:3.3}

We start by recalling some basic facts concerning $\Gamma$-convergence theory and relative entropies. The claims in the following three remarks are easy to prove.

\begin{remark}
\label{r:Glim}
  The $\Gamma$-liminf and $\Gamma$-limsup of $(I_n)$ are lower
  semicontinuous functionals, coercive if $(I_n)$ is
  equicoercive.
  
Let $J \colon X \to [0,+\infty]$. Then
\begin{itemize}
\item[(i)]{If for each sequence $x_n\to x$, $\varliminf_n I_n(x_n)\ge J(x)$, then $J \le \Gliminf_n I_n$.
}

\item[(ii)]{If there exists a sequence $x_n\to x$ such that $\varlimsup_n I_n(x_n)\le J(x)$, then $J \ge \Glimsup_n I_n$.
}
\end{itemize}
The $\Gliminf_n I_n$ and $\Glimsup_n I_n$ are respectively the smallest
and the largest lower semicontinuous functionals on $X$ satisfying conditions (i) and (ii) above.

Moreover for each $x\in X$, open set $O\subset X$, compact $K \subset X$ the following holds
\begin{itemize}
\item[(a)]{There exists a sequence $x_n\to x$ such that
\begin{equation*}
\varlimsup_n I_n(x_n)\le \big(\Glimsup_n I_n \big)(x)
\end{equation*}
and
\begin{equation*}
\varlimsup_n \inf_{y\in O} I_n(y)
  \le  \inf_{y\in O} \big(\Gliminf_n I_n\big)(x)
\end{equation*}
}

\item[(b)]{For each sequence $x_n\to x$ 
\begin{equation*}
\varliminf_n I_n(x_n)\ge \big(\Gliminf_n I_n\big)(x)
\end{equation*}
and
\begin{equation*}
\varliminf_n \inf_{y\in K} I_n(y)\le  
 \inf_{y\in K} \big(\Gliminf_n I_n\big)(y)
\end{equation*}
Additionally, if $I_n$ is equicoercive then for each closed set $C\subset X$
\begin{equation*}
\varliminf_n \inf_{y\in C} I_n(y)\le  
 \inf_{y\in C} \big(\Gliminf_n I_n\big)(x)
\end{equation*}
}
\end{itemize}
\end{remark}

Hereafter for $\mu \in \mathcal{P}(X)$ and $A$ a Borel subset of $X$ such that $\mu(A)>0$, $\mu^{A} \in \mathcal{P}(X)$ denotes the probability measure obtained by conditioning $\mu$ on $A$.
\begin{remark}
\label{r:relentr}
If $\mf G \subset \mf F$ then
\begin{equation}
\label{e:H0}
H_{\mf G}(\nu|\mu) \le H_{\mf F}(\nu|\mu)
\end{equation}

If $\mf F$ is the Borel $\sigma$-algebra of $X$, the supremum over
$\varphi$ in \eqref{e:H1} can equivalently run over the test functions
$\varphi \in L_1(X,d\nu)$, or equivalently over $\varphi \in
C_\mathrm{c}(X)$, or equivalently over the set of measurable functions
$\varphi$ taking only a finite number of values. Moreover
\begin{equation}
\label{e:Hdef}
  H(\nu|\mu)=
  \begin{cases}
    \int_X\!\mu(dx)\, \frac{d\nu}{d\mu}(x) \log  \frac{d\nu}{d\mu}(x)
    & \text{if $\nu <<\mu$}
\\
+\infty & \text{otherwise}
  \end{cases}
\end{equation}
In particular, for $\nu(A)>0$, $\tfrac{d\nu^A}{d\mu}= \tfrac{\un {A}}{\nu(A)} \tfrac{d\nu}{d\mu} $, so that
\begin{equation}
  \label{e:Hcond}
\begin{split}
  H(\nu^{A}|\mu) = & 
  -\log \nu(A)+\tfrac{1}{\nu(A)} \int_A  d\nu(x) \,\log \frac{d\nu}{d\mu}(x) 
\\ 
\le &
-\log \nu(A) + \frac{1}{\nu(A)} 
\left(H(\nu|\mu)+(1-\nu(A))\log\left(\tfrac{1-\nu(A)}{1-\mu(A)} \right)\right)
\\
\le & -\log \nu(A) + \frac{1}{\nu(A)} H(\nu|\mu)+ 1-\tfrac{\mu(A)}{\nu(A)}
\end{split}
\end{equation}
where the first inequality follows by taking $\varphi$ constant on $A^c$ in the definition \eqref{e:H1}.

Let $Y$ be also a Polish space, $\lambda \in \mc P(Y)$, $\theta \colon X \to Y$ measurable and $\mf F_\theta$  the associated $\sigma$-algebra. If $\lambda \in \mc P(Y)$ then
\begin{equation}
\label{e:Hproj}
H(\lambda | \mu \circ \theta^{-1}) =\inf_{\nu:\, \nu \circ \theta^{-1}=\lambda} H(\nu|\mu)=H_{\mf F_\theta}(\nu|\mu) \quad \text{for all $\nu \,:\: \nu \circ \theta^{-1}=\lambda$}
\end{equation}
If $H(\lambda |\mu \circ \theta^{-1}) <+\infty$ then the infimum in \eqref{e:Hproj} is attained, namely
\begin{equation}
\label{e:Hproj2}
\begin{split}
& H(\lambda |\mu \circ \theta^{-1})= H(\bar \nu |\mu)
\\
& \bar \nu(dx) = \mu(dx)\, \bb E^{\nu}  \Big(\frac{d\nu}{d\mu} \Big| \mf F_\theta \Big)(x)
 \quad \text{for all $\nu\,:\: \nu \circ \theta^{-1}=\lambda$}
\end{split}
\end{equation}
\end{remark}

\begin{remark}
\label{r:relentr2}
If $\mf G=\sigma((E^i)_{i=0}^N)$ is a
$\sigma$-algebra generated by a finite partition of $X$, then
\begin{equation*}
  H_{\mf G}(\nu|\mu)
 = \sum_{i=0}^N \nu(E^i) \log \frac{\nu(E^i)}{\mu(E^i)}
 \end{equation*}
where we understand $\nu(A) \log \tfrac{\nu(A)}{\mu(A)}=0$
 whenever $\nu(A)=0$ and $\nu(A) \log
 \tfrac{\nu(A)}{\mu(A)}=+\infty$ if $\mu(A)=0$ but
 $\nu(A)>0$. 
 
Moreover taking $\varphi=\log(1+\mu(A))\un {A}$ in \eqref{e:H1} one obtains
\begin{equation}
   \label{e:H4}
   \nu(A) \le 
   \frac{\log 2 + H(\nu|\mu)}{\log(1+ \tfrac{1}{\mu(A)})}
 \end{equation}
whenever $H(\nu|\mu)<+\infty$ and $\mu(A)>0$.
\end{remark}

\begin{remark}
\label{r:E}
Let $\mu,\,\nu\in \mathcal{P}(X)$, and $(K_\ell)_{\ell \in \bb N}$ a sequence of compacts subsets of $X$ such that $\lim_{\ell} \mu(K_\ell)=1$.
Then for each $\delta>0,\,\ell \in \bb N$ there exists a finite
family $(E^i_{\delta,\ell})_{i=1}^{N_{\delta,\ell}}$ of Borel subsets of $X$ such that:
\begin{itemize}
\item[(i)] $\cup_i E^i_{\delta,\ell} \supset K_\ell$ and $E^i_{\delta,\ell} \cap E^{i'}_{\delta,\ell}= \emptyset$ if $i\neq i'$.
\item[(ii)] $\mathrm{diameter}(E^i_{\delta,\ell})\le \delta$, for $i=1,\ldots,\,N_{\delta,\ell}$.
\item[(iii)]  $\mu(\partial E^i_{\delta,\ell}) = \nu(\partial E^i_{\delta,\ell}) =0$, for $i=1,\ldots,\,N_{\delta,\ell}$.
\item[(iv)] Each $E^i_{\delta,\ell}$ has nonempty interior.
\end{itemize}
Set $E^0_{\delta,\ell}=X \setminus \cup_{i \ge 1} E^{i}_{\delta,\ell}$. One may also assume, with no loss of generality
\begin{itemize}
\item[(v)] The partition $(E^i_{\delta,\ell})_{i=0}^{N_{\delta,\ell}}$ is finer than
$(E^i_{\delta',\ell'})_{i=0}^{N_{\delta',\ell'}}$ if $\delta \le
\delta'$ and $\ell \ge \ell'$.
\end{itemize}
Moreover, if $\mf G_{\delta,\ell}$ is the $\sigma$-algebra generated by $(E^i_{\delta,\ell})$
\begin{equation*}
  \lim_{\ell} \lim_{\delta} H_{\mf G_{\delta,\ell}}(\nu|\mu)= H(\nu|\mu)
\end{equation*}
\end{remark}
\begin{proof}
Fix $\delta,\,\ell$ and take a finite cover of $K_\ell$ with open balls $B_{\delta/2}(x_i)$ of radius $\delta/2$ and centered at $x_i \in K_\ell$, $i=1,\ldots,\,N_{\ell,\delta}$. Take $r>0$ such that $r\le \delta/2$ and $r \le \mathrm{distance}(x_i,x_j)$ for all $i\neq j$. By $\sigma$-additivity of $\mu$ and $\nu$, there exists $\delta' \in ]\delta/2,\delta/2+r[$ such that $\mu(\partial B_{\delta'}(x_i))=\nu(\partial B_{\delta'}(x_i))=0$ for all $i$. Then take
\begin{equation*}
\begin{cases}
E^1_{\delta,\ell}=B_{\delta'}(x_1) &
\\
E^{i}_{\delta,\ell}=B_{\delta'}(x_i) \setminus \cup_{j<i} E^{j}_{\delta,\ell}  & \text{for $i>1$.}
\end{cases}
\end{equation*}
It is immediate to check that $(E^i_{\delta,\ell})$ satisfies (i)-(iv); and by a refining procedure one gets the $E^i_{\delta,\ell}$ to satisfy (v) as well. The convergence of the relative entropies is a consequence of (i)-(v).
\end{proof}

Next we turn to the proofs of the statements in section~\ref{ss:3.1} and \ref{ss:3.2}.
\begin{proof}[Proof of Proposition~\ref{p:tight}]
\noindent (H)$\Rightarrow$(P).
$\mu_n$ is in the $0$ sublevel set of $H_n$, and thus $(\mu_n)$ is precompact (and tight) by the definition of equicoercivity.

\smallskip
\noindent (P)$\Rightarrow$(H).
Let $(\nu_n)$ be a sequence in $\mathcal{P}(X)$ such that $\varlimsup_n H(\nu_n|\mu_n)<+\infty$. Since $\mu_n$ is tight, there exists an increasing sequence $(K_\ell)$ of compacts  such that $\lim_\ell \varlimsup_n  \mu_n(K_\ell^c)=0$. Since $H(\nu_n|\mu_n)$ is uniformly bounded, the application of \eqref{e:H4} with $A=K_\ell^c$ yields $\lim_\ell \varlimsup_n \nu_n(K_\ell^c)=0$. Namely $(\nu_n)$ is tight.

\smallskip
\noindent (P)$\Leftrightarrow$(L). It is trivial.
\end{proof}

 \begin{proof}[Proof of Proposition~\ref{p:conv}]
  \noindent  (P1)$\Rightarrow$(P2).
  Assume $\mu_n \to \mu$ and let $(E^i_{\delta,\ell})_{i=0}^{N_{\delta,\ell}}$ be as in Remark~\ref{r:E} with $\nu=\mu$. Let $(\varphi_n)$ be as in the
  statement (P2), and define 
  $\varphi_{n,\delta,\ell},\,\varphi_{\delta,\ell}\colon X \to \bar{\bb
    R}$ by
\begin{eqnarray*}
& & \varphi_{n,\delta,\ell}(x)= \inf_{y \in E^i_{\delta,\ell}}
\varphi_n(y) \qquad \text{if $x\in E^i_{\delta,\ell}$}
\\
& & \varphi_{\delta,\ell}(x)= \varliminf_n \varphi_{n,\delta,\ell}(x)
\end{eqnarray*}
Note that by Remark~\ref{r:E}-(iii) and -(v)
\begin{equation*}
\mu\big(\cup_{\ell,\delta>0} \cup_{i=1}^{N_{\delta,\ell}} \partial E_{\delta,\ell}^i\big) =\lim_{\ell} \lim_{\delta} \sum_{i=0}^{N_{\delta,\ell}} \mu(E^i_{\delta,\ell})=0
\end{equation*}
On the other hand, if $x \not \in \cup_{\ell,\delta>0} \cup_{i=1}^{N_{\delta,\ell}} \partial E_{\delta,\ell}^i$ it is easy to check
 \begin{equation*}
 \lim_\ell \lim_\delta \varphi_{\delta,\ell}(x) = (\Gliminf_n
\varphi_n)(x)
 \end{equation*}
 the limit being monotone increasing by Remark~\ref{r:E}-(v). Thus
$\lim_\ell \lim_\delta \varphi_{\delta,\ell} = (\Gliminf_n \varphi_n)$ $\mu$-a.e., and by monotone convergence
\begin{equation*}
  \begin{split}
       \mu (\Gliminf_n \varphi_n) & =
      \lim_\ell \lim_\delta \mu (\varphi_{\delta,\ell})
      =
      \lim_\ell \lim_\delta \sum_{i=0}^{N_{\delta,\ell}} 
      \big[ \mu(E^i_{\delta,\ell}) \,
      \varliminf_n \inf_{y \in E_{\delta,\ell}^i} 
       \varphi_{n,\delta,\ell}(y) \big]
      \\ & 
      =
      \lim_\ell \lim_\delta \varliminf_n
      \sum_{i=0}^{N_{\delta,\ell}} \big[  \mu_n(E^i_{\delta,\ell}) 
      \inf_{y \in E^i_{\delta,\ell}} 
           \varphi_{n,\delta,\ell}(y) \big]
      \le \varliminf_n \mu_n\big(\varphi_n \big)
    \end{split}
\end{equation*}
where last equality follows from Remark~\ref{r:E}-(iii).

\smallskip
 \noindent (P2)$\Rightarrow$(H).
By (P2), $\mu_n(\varphi) \to \mu(\varphi)$ for each $\varphi \in C_\mathrm{b}(X)$. Let now $\nu \in \mathcal{P}(X)$ and let $(\nu_n)$ be an arbitrary sequence
in $\mathcal{P}(X)$ such that $\nu_n\to \nu$. Then
\begin{equation*}
  \begin{split}
  \varliminf_n  H(\nu_n|\mu_n) & =
\varliminf_n \sup_{\varphi \in C_{\mathrm{b}}(X)} \big\{\nu_n(\varphi)
      -\log \mu_n(e^\varphi) \big\} 
 \ge  \sup_{\varphi \in C_{\mathrm{b}}(X)} \varliminf_n\big\{\nu_n(\varphi)
      -\log \mu_n(e^\varphi) \big\} 
\\ &   =  \sup_{\varphi \in C_{\mathrm{b}}(X)} \big\{\nu(\varphi)
     -\log \mu (e^\varphi) \big\}=H(\nu|\mu)
   \end{split}
\end{equation*}
Namely $\varliminf_n H_n(\nu_n)\ge H(\nu)$, and thus the
$\Gamma$-liminf inequality holds. It is enough to prove the
$\Gamma$-limsup inequality for $\nu$ such that
$H(\nu|\mu)<+\infty$. In particular $\nu$ is absolutely continuous
with respect to $\mu$. Let
$(E^i_{\delta,\ell})_{i=0}^{N_{\delta,\ell}}$ be as in
Remark~\ref{r:E}. Fix $\delta,\,\ell>0$, and for $n$
large enough define the probability $\nu_{n,\delta,\ell} \in \mathcal{P}(X)$ as
\begin{equation*}
\label{e:nundl}
  \nu_{n,\delta,\ell}(A)=\sum_{i=0}^{N_{\delta,\ell}} 
  \nu(E^i_{\delta,\ell})
  \frac{\mu_n(A \cap E^i_{\delta,\ell})
  }{\mu_n(E^i_{\delta,\ell})}
\end{equation*}
$ \nu_{n,\delta,\ell}$ is well defined since
$\nu(E^i_{\delta,\ell})=0$ whenever
$\mu_n(E^i_{\delta,\ell})=0$ for $n$ large enough. Then
\begin{equation*}
  \lim_\ell \lim_\delta \lim_n \nu_{n,\delta,\ell} = \nu
\end{equation*}
On the other hand, by explicit calculation,
$H(\nu_{n,\delta,\ell}|\mu_n)= H_{\mf
  G_{\delta,\ell}}(\nu|\mu_n)$, and recalling that the sets $E^i_{\delta,\sigma}$ are $\mu$- and $\nu$-regular
\begin{equation*}
  \begin{split}
    \varlimsup_n H(\nu_{n,\delta,\ell}|\mu_n) & 
     = \varlimsup_n H_{\mf G_{\delta,\ell}}(\nu|\mu_n) 
     = \varlimsup_n \sum_{i=0}^{N_{\delta,\ell}} \nu(E^i_{\delta,\ell}) 
         \log \frac{\nu(E^i_{\delta,\ell}) }{\mu_n(E^i_{\delta,\ell}) }
\\ &  = \sum_{i=0}^{N_{\delta,\ell}} \nu(E^i_{\delta,\ell}) 
           \log \frac{\nu(E^i_{\delta,\ell}) }{\mu(E^i_{\delta,\ell}) } 
= H_{\mf G_{\delta,\ell}}(\nu|\mu) 
\le H(\nu|\mu)
  \end{split}
\end{equation*}
Thus there exist sequences $(\delta_n)$, $(\ell_n)$ such
that $\nu_n:= \nu_{n,\delta_n,\ell_n} \to \nu$ and $\varlimsup_n
H(\nu_n|\mu_n) \le H(\nu|\mu)$.

\smallskip
 \noindent (H)$\Rightarrow$(P1). $\mu_n$ is the unique minimizer of $H_n$, and $\mu$ is the unique minimizer of
$H$. Since converging sequences of minimizers converge to minimizers of the $\Gamma$-limit, see \cite[Proposition~7.18]{DM}, one is left to show that $(\mu_n)$ is precompact in $\mc P(X)$, namely that 
\begin{equation}
\label{e:muntight}
\sup_{K \in \mc K(X)} \varliminf_n \mu_n(K)=1
\end{equation}
(H) implies that there exists a sequence $(\nu_n)$ converging to $\mu$ such that $\varlimsup_n H(\nu_n|\mu_n) \le H(\mu|\mu)$; so that, in view of the tightness of $(\nu_n)$
\begin{equation}
\label{e:nunrec}
\lim_n  H(\nu_n|\mu_n)=0, \qquad \qquad \sup_{K\in \mc K(X)} \varliminf_n \nu_n(K)=1
\end{equation}
Reversing the inequality in \eqref{e:H4} (with $A=K$), one gets for each $K\in \mc K(X)$
\begin{equation*}
\mu_n(K) \ge \frac 1{2^{\tfrac{1}{\nu_n(K)} }\exp[H(\nu_n|\mu_n)]-1}
\end{equation*}
Taking the liminf in $n$ and the supremum over $K\in \mc K(X)$, one gets \eqref{e:muntight} by \eqref{e:nunrec}.

\smallskip
 \noindent (P1)$\Rightarrow$(L1). Fix $\eps>0$, $K\in \mf K(X)$, and let $K^\eps$ be the open $\eps$-enlargement of $K$ with respect to any compatible metric on $X$. For each $n$, take a compact $K_{n,\eps}$ such that $K\subset K_{n,\eps} \subset K^\eps$ and $\mu_n(K_{n,\eps})\ge \mu_n(K^\eps)-\eps$. Then by (P1)
 \begin{equation*}
\varliminf_{\eps \downarrow 0} \varliminf_n \mu_n(K_{n,\eps}) \ge  \varliminf_{\eps \downarrow 0}  \big(\varliminf_n \mu_n(K^\eps) -\eps\big) \ge \varliminf_{\eps \downarrow 0} \mu(K^\eps) -\eps =\mu(K)
 \end{equation*} 
 Thus there exists $\eps_n \downarrow 0$ such that $K_n:=K_{n,\eps_n} \to K$ in $\mf K(X)$ and $\varlimsup_n l_{a,\mu_n}(K_n) \le l_{a,\mu}(K)$. Namely the $\Gamma$-limsup inequality holds.
 
Fix now $K\in \mf K(X)$ and let $(K_n)$ be a sequence converging to $K$ in $\mf K(X)$. Define $Q_m:=K \bigcup \cup_{n \ge m} K_n$. Then $Q_m$ is compact for all $m$, and by (P1)
 \begin{equation*}
\varlimsup_n \mu_n(K_n) \le \lim_m \varlimsup_n \mu_n(Q_m) \le \lim_m \mu(Q_m) = \mu(K)
 \end{equation*}
 which is the $\Gamma$-liminf inequality for $(l_{a,\mu_n})$.

\smallskip
 \noindent (L2)$\Rightarrow$(P1). Let $\eps>0$ and $O \subset X$ be open. By the regularity of $\mu$ on $X$, there exists a compact $K \subset O$ such that $\mu(K) \ge \mu(O)-\eps$. By the $\Gamma$-limsup inequality for $l_{a,\mu_n}$ and Remark~\ref{r:Glim}-(a), there exists a sequence $(K_n)$ in $\mf K(X)$ such that $\varliminf_n \mu_n(K_n) \ge \mu(K)$. Since $K$ is compact $O \supset K^\delta$ for some $\delta>0$, so that for $n$ large enough $K_n \subset O$. Thus
 \begin{equation*}
 \varliminf_n \mu_n(O) \ge  \varliminf_n \mu_n(K_n) \ge \mu(K) \ge \mu(O)-\eps
 \end{equation*}
and we conclude since $\eps>0$ was arbitrary.

\smallskip
 \noindent (L3)$\Rightarrow$(P1). By sequential compactness of $\Gamma$-convergence \cite[Chap.~10]{DM}, from any subsequence $(\mu_{n'})$ of $(\mu_n)$ one can extract a further subsequence $\mu_{n''}$ such that $l_{a,\mu_{n''}}$ $\Gamma$-converges to $l_{a,\mu}$, thus $\mu_{n''} \to \mu$ by the statement (L2)$\Rightarrow$(P1) proved above. Since $\mc P(X)$ is Polish, the Urysohn property holds, and $\mu_n \to \mu$. 
 
\smallskip
 \noindent (L1)$\Rightarrow$(L2) and (L1)$\Rightarrow$(L3) are trivial.
\end{proof}

\begin{proof}[Proof of Theorem~\ref{t:tightld}]
 \noindent (P) $\Rightarrow$ (H).
 By (P), for each $\ell>0$, there exists a compact $K_\ell \subset X$ such that
  $\mu_n(K_\ell^c) \le e^{-\ell\, a_n}$. By \eqref{e:H4}, for each
  $\nu \in \mathcal{P}(X)$
\begin{equation*}
  \nu(K_\ell^c) 
  \le \frac{\tfrac{1}{a_n} H(\nu|\mu_{n})+\tfrac{\log 2}{a_n} }
         {\tfrac{1}{a_n} \log \big(1+\frac{1}{\mu_{n} (K_\ell^c)} \big)}
  \le \frac{H_n^{\mb a}(\nu) +\tfrac{\log 2}{a_n} }{\ell}          
\end{equation*}
Let $n_0$ be such that $a_n \ge 1$ for $n \ge n_0$. Then for $M>0$
\begin{equation*}
\cup_{n \ge n_0}\big\{\nu \in \mathcal{P}(X)\,:\:H_n^{\mb a}(\nu) \le M \big\}
 \subset \left\{\nu \in
\mathcal{P}(X)\,:\: \forall \ell>0,\,\nu(K_\ell^c) 
\le \frac{M+\log 2}{\ell}\right\}
\end{equation*}
which is a tight set, and thus precompact in $\mathcal{P}(X)$. Since $\cup_{n < n_0}\big\{\nu \in \mathcal{P}(X)\,:\:H_n^{\mb a}(\nu) \le M \big\}$ is precompact, we conclude.

\smallskip
 \noindent (H) $\Rightarrow$ (P).
 Note that, by \eqref{e:Hcond}, for each $\ell>0$ and integer $n_0\ge 1$
\begin{equation*}
\begin{split}
  \mc P_{n_0,\ell} :=&  \cup_{n \ge n_0}
  \big\{ \mu^{K^c}_n,\, \text{$K \subset X$ is
    compact and } \mu_n(K^c) \ge e^{-\ell\,a_n} \big\}
\\
 \subset &  \cup_{n \ge n_0} \{\nu \in \mathcal{P}(X)\,:\:H_n^{\mb a}(\nu) \le
  \ell\}
\end{split}
\end{equation*}
Therefore by (H), for each $\ell>0$ there exists $n_0(\ell)$
such that $\mc P_{n_0(\ell),\ell}$ is precompact in $\mathcal{P}(X)$, and
thus tight. In particular, for each $\ell>0$ there exists a compact
set $K_{\ell} \subset X$ such that $\mu_n^{K^c}(K_{\ell}^c)
\le 1/2$, for each $n \ge n_0(\ell)$ and each compact $K$ such
that $\mu_n(K^c)\ge \exp(-\ell\,a_n)$. But $\mu_n^{K^c}(K^c)=1$ for each compact $K$ with $\mu_n(K^c)>0$. Thus $K_{\ell} \neq K$ for each compact $K$ such that $\mu_n(K^c) \ge \exp(-\ell\,a_n)$ for some $n \ge n_0(\ell)$. Namely $\mu_n(K_{\ell}^c) \le \exp(-\ell\,a_n)$ for each $\ell>0$ and $n \ge n_0(\ell)$.

\smallskip
 \noindent (P) $\Leftrightarrow$ (L). It is trivial.
\end{proof}

\begin{proof}[Proof of Theorem~\ref{t:convldlow}]
\noindent  (P1) $\Rightarrow$ (H1).
For $x \in X$ and $\delta>0$ let $B_\delta(x)$ the open ball of radius $\delta$ centered at $x$. Fix $n$ and define $\nu_{n,\delta} \in \mathcal{P}(X)$ by
\begin{equation*}
\nu_{n,\delta}:=
\begin{cases}
\mu_n^{B_\delta(x)}  & \text{if $\mu_n(B_\delta(x))>0$}
\\
\delta_x & \text{otherwise}
\end{cases}
\end{equation*}
and note $H(\nu_{n,\delta}|\mu_n)= - \log \mu_n(B_\delta(x))$, where we
understand $-\log(0)=+\infty$. By (P1), for each $\delta>0$
\begin{equation*}
  \varlimsup_n H_n^{\mb a}(\nu_{n,\delta})= 
   - \varliminf_n \tfrac{1}{a_n} \log \mu_n(B_\delta(x)) 
  \le \inf_{y \in B_\delta(x)} I(y) \le I(x)
\end{equation*}
On the other hand $\lim_\delta \lim_n \nu_{n,\delta}= \delta_x$ in
$\mathcal{P}(X)$. In particular, by a diagonal argument, there exists a sequence $(\delta_n)$ converging to $0$ (slowly enough)
such that $\lim_n \nu_{n,\delta_n} = \delta_x$ and $\varlimsup_n
H_n^{\mb a}(\nu_{n,\delta_n}) \le I(x)$. (H1) follows by Remark~\ref{r:Glim}-(ii).

\smallskip
\noindent (H1) $\Rightarrow$ (P2).
Let $Y:=\big\{x \in X\,:\: (\Gliminf_n \varphi_n)(x)>-\infty \big\}$. By the definition of the
$\Gamma$-liminf, for each $x \in Y$ there exist $\delta(x)>0$ and
$n_0(x)\in \bb N$ such that
\begin{equation}
\label{e:glinfbelow}
\inf_{y \in B_{\delta(x)}} \inf_{n\ge n_0(x)} \varphi_n(y)>
-\infty
\end{equation}
For $x \in Y$, let $(\nu_{n,x})$ be a sequence converging to
$\delta_x$ in $\mathcal{P}(X)$ and such that $\varlimsup H_n^{\mb a}(\nu_{n,x})
\le I(x)$. Such a sequence exists by (H1). By
\eqref{e:Hcond}, it is easily seen that $(\nu_{n,x})$ can be assumed
to be concentrated on $B_{\delta(x)}(x)$. By \eqref{e:H1}
\begin{equation}
\label{e:relentrineq5}
\log \mu_n(e^\varphi) \ge -H(\nu_{n,x}|\mu_n) 
+ \nu_{n,x}(\varphi)
\end{equation}
for each measurable $\varphi\colon X\to [-\infty,+\infty]$, provided
we read the right hand side as $-\infty$ whenever
$H(\nu_{n,x}|\mu_n)=+\infty$ or
$\nu_{n,x}(\varphi^-)=+\infty$. Evaluating \eqref{e:relentrineq5} for
$\varphi=a_n \varphi_n$, taking the liminf in $n$ and next optimizing
on $x \in Y$
\begin{equation*}
\varliminf_n \tfrac{1}{a_n}\log \bb \mu_n\big(\exp(a_n \varphi_n)\big)
 \ge \sup_{x \in Y} \Big\{- \varlimsup_n H_n^{\mb a}(\nu_{n,x}) 
 + \varliminf_n \nu_{n,x}(\varphi_n)  \Big\}
\end{equation*}
Since $\nu_{n,x}$ is concentrated on $B_{\delta(x)}(x)$, and by
\eqref{e:glinfbelow} $\varphi_n$ is bounded from below on
$B_{\delta(x)}(x)$ for $n\ge n_0(x)$, (H1) and Proposition~\ref{p:conv} yield
\begin{equation*}
\begin{split}
\varliminf_n \tfrac{1}{a_n}\log \bb \mu_n\big(\exp(a_n \varphi_n)\big) 
& \ge \sup_{x \in Y} \big\{- I(x) 
 + (\Gliminf_n \varphi_n )(x) \big\}
\\
& =
\sup_{x \in X} \big\{- I(x) 
 + (\Gliminf_n \varphi_n )(x) \big\}
\end{split}
\end{equation*}

\smallskip
\noindent
(P2) $\Rightarrow$ (P1). Fix an open set $O \subset X$ and $M>0$. Then $\varphi_n \equiv M \un {O}$ is lower semicontinuous, and thus coincides with its $\Gamma$-limit. It follows
\begin{equation}
\label{e:ineq6}
\begin{split}
  \tfrac{1}{a_n} \log \mu_n(\exp(a_n\,\varphi_n)) & =
  \tfrac{1}{a_n} \log \left(1+\mu_n(O)\exp(a_n\,M))\right)
\\
 &\le\tfrac{\log 2}{a_n}+ 
 \max(0,M+\tfrac{1}{a_n} \log (\mu_n(O))
\end{split}
\end{equation}
By (P2) applied to such a sequence $\varphi_n$, one gathers taking the limit in \eqref{e:ineq6}
\begin{equation*}
\max(-M,\varliminf_n \tfrac{1}{a_n} \log (\mu_n(O))
\ge -M+ \sup_{x \in X} \big(M \un {O}-I(x) \big) 
 \ge -\inf_{x \in O} I(x).
\end{equation*}
This implies (P1) when taking $M\to \infty$.

\smallskip
\noindent (H2) $\Rightarrow$ (H1). Take $\nu=\delta_x$.

\smallskip
\noindent (H1) $\Rightarrow$ (H2).
Since $H_n^{\mb a}$ is a convex functional, $\Glimsup_n H_n^{\mb a}$ is also convex. For an arbitrary $\nu \in \mathcal{P}(X)$, by Jensen inequality and (H1)
\begin{equation*}
\begin{split}
  (\Glimsup_n H_n^{\mb a})(\nu) & = 
 (\Glimsup_n  H_n^{\mb a})\Big(\int_{\mathcal{P}(X)} \nu(dx)\,\delta_x \Big)
\\
  & \le \int_{\mathcal{P}(X)}\! \nu(dx)\, (\Glimsup_n H_n^{\mb a})(\delta_x) \le 
\int_{\mathcal{P}(X)} \!\nu(dx)\, I(x)
\end{split}
\end{equation*}

\smallskip
\noindent (P1) $\Rightarrow$ (L). Fix $\eps>0$, $K\in \mf K(X)$, and let $K^\eps$ be the open $\eps$-enlargement of $K$ with respect any fixed compatible metric on $X$. Then, by the regularity of $\mu_n$ on $X$, for each $n$ there exists $K_{n,\eps} \subset K^\eps$ compact such that $\mu_n(K_{n,\eps}) \ge  \exp(-\eps\,a_n)\,\mu_n(K^\eps)$. By (P1)
 \begin{equation*}
\varlimsup_{\eps \downarrow 0} \varlimsup_n l^{\mb a}_n(K_{n,\eps}) \le  \varlimsup_{\eps \downarrow 0}  \varlimsup_n -
\tfrac{1}{a_n} \log \mu_n(K^\eps) -\eps \le \varlimsup_{\eps \downarrow 0} \inf_{x\in K^\eps} I(x)-\eps =\inf_{x\in K}I(x)
 \end{equation*} 
 Thus there exists $\eps_n \downarrow 0$ such that $K_n:=K_{n,\eps_n} \cup K$ converges to $K$ in $\mf K(X)$ and $\lim_n l^{\mb a}_n(K_n) \le l_I(K)$. Namely the $\Gamma$-limsup inequality holds by Remark~\ref{r:Glim}-(ii).

\smallskip
\noindent (L) $\Rightarrow$ (P1). Let $\eps>0$ and $O \subset X$ open. Since $I$ is lower semicontinuous, there exists $K \subset O$ compact such that $\inf_{x\in K} I(x) \le \inf_{x \in O}I(x)+\eps$. By (L) there exists a sequence $(K_n)$ converging to $K$ in $\mf K(X)$ such that
$\varlimsup_n  -\tfrac{1}{a_n} \log \mu_n(K_n) \le \inf_{x\in K} I(x)$. Since for $n$ large enough $K_n \subset O$
\begin{equation*}
\varlimsup_n  -\tfrac{1}{a_n} \log \mu_n(O) \le 
\varlimsup_n  -\tfrac{1}{a_n} \log \mu_n(K_n) \le \inf_{x\in K} I(x)
\le \inf_{x\in K} I(x)+\eps
\end{equation*}
and (P1) follows since $\eps>0$ is arbitrary.
\end{proof}

\begin{proof}[Proof of Theorem~\ref{t:convldup}]
\noindent (P1) $\Rightarrow$ (H1). Let $x \in X$ and $(\nu_n) \subset \mathcal{P}(X)$ be such that $\lim_n \nu_n =\delta_x$ in $\mathcal{P}(X)$.  In view of Remark~\ref{r:Glim}-(i), it is enough to show $\varliminf_n H_n^{\mb a}(\nu_n)\ge I(x)$. Fix $\eps>0$; since $(\nu_n)$ is tight, there exists $K\subset X$ compact such that $\nu_n(K)\ge 1-\eps$ for all $n$. By  \eqref{e:H4}, for each Borel set $A \subset X$
\begin{equation*}
H^{\mb a}_n(\nu_n) \ge \frac{\nu_n(A) }{a_n}
         \log\Big(1+\frac{1}{\bb \mu_n(A)} \Big) -\frac{\log 2}{a_n}
         \ge -\frac{\nu_n(A) }{a_n}
         \log \mu_n(A) -\frac{\log 2}{a_n}
\end{equation*}
Take now $A=K\cap \upbar{B}_\eps(x)$, where $\upbar{B}_\eps(x)$ is the closed ball of radius $\eps$ centered at $x$. Note that $A$ is compact and $\varliminf_n \nu_n(A)\ge 1-\eps$, thus by (P1)
\begin{equation*}
  \varliminf_n H_n^{\mb a}(\nu_n) 
 \ge - (1-\eps)\varlimsup_n \tfrac{1}{a_n}\log \mu_n(A) 
     \ge (1-\eps) \inf_{y \in A} I(y) \ge (1-\eps) \inf_{y\in \upbar B_\eps(x)}I(y)
\end{equation*}
Since $\eps>0$ was arbitrary, one can take the limit
$\eps \downarrow 0$ in the above formula, and since $I$ is lower semicontinuous the right hand side in the above formula converges to $I(x)$.

\smallskip
\noindent (H2) $\Rightarrow$ (H1). Take $\nu=\delta_x$.

\smallskip
\noindent (H1) $\Rightarrow$ (H2). 
 Assume (H1). Let $\nu \in \mathcal{P}(X)$ and $(\nu_n)$ be a sequence
converging to $\nu$ in $\mathcal{P}(X)$. One needs to show $\varliminf_n
H_n^{\mb a}(\nu_n) \ge \nu(I)$.

For $\delta,\ell>0$ let
$(E^i_{\delta,\ell})_{i=0}^{N_{\delta,\ell}}$ be as in
Remark~\ref{r:E} (with $\mu=\nu$).
For $i\in \{0,\ldots,N_{\delta,\ell}\}$ such that
$\nu_n(E_{\delta,\ell}^i)>0$ define the probability measures
$\nu^i_{n,\delta,\ell}:=\nu_n^{E_{\delta,\ell}^i} \in \mathcal{P}(X)$. Then
by \eqref{e:Hdef}, for each $n,\,\ell>0$
\begin{equation}
\label{e:Hineq4}
\begin{split}
H(\nu_n|\mu_n) & = 
\sum_{i=0}^{N_{\delta,\ell}} \nu_n(E_{\delta,\ell}^i) 
                  H(\nu^i_{n,\delta,\ell}|\mu_n)
+ \nu_n(E_{\delta,\ell}^i) \log  \nu_n(E_{\delta,\ell}^i)
\\
& \ge \sum_{i=0}^{N_{\delta,\ell}}  \nu_n(E_{\delta,\ell}^i) 
      H(\nu_{n,\delta,\ell}^i|\mu_n) 
-\log N_{\delta,\ell}
\end{split}
\end{equation}
where the terms in the above sums are understood to vanish for all $i$
such that $\nu_n(E_{n,\ell}^i)=0$. Dividing \eqref{e:Hineq4} by $a_n$,
taking the liminf and recalling that the sets
$E_{\delta,\ell}^i$ are $\nu$-regular
\begin{equation*}
\begin{split}
\varliminf_n H_n^{\mb a}(\nu_n) & \ge 
\sum_{i=0}^{N_{\delta,\ell}} \varliminf_n \nu_n(E_{\delta,\ell}^i)
                  H_n^{\mb a}(\nu^i_{n,\delta,\ell})
\\
& = 
\sum_{i=0}^{N_{n,\ell}} \nu(E_{\delta,\ell}^i) 
               \varliminf_n  H_n^{\mb a}(\nu^{i}_{n,\delta,\ell})
=\int I_{\delta,\ell}(x) d\nu(x)
\end{split}
\end{equation*}
where $I_{\delta,\ell}$ is defined by
\begin{equation*}
  I_{\delta,\ell}(x):= \varliminf_n H_n^{\mb a}(\nu^{i}_{n,\delta,\ell})
\qquad \text{if $x \in E^i_{\delta,\ell}$}
\end{equation*}
Note that $I_{\delta,\ell}$ is monotone both in $\delta$ and $\ell$, the partitions $\{E_{\delta,\ell}^i\}$ are
increasing as $\delta\downarrow 0$ and $\ell \uparrow +\infty$, see \eqref{e:H1}. By
monotone convergence
\begin{equation*}
  \varliminf_n H_n^{\mb a}(\nu_n) \ge \int  \big(\lim_\ell \lim_\delta I_{\delta,\ell}(x)\big) d\nu(x)  
  \end{equation*}
However, since $\lim_\ell \lim_\delta \lim_n \nu_n=\delta_x$, (H1) implies $\lim_{\ell} \lim_\delta I_{\delta,\ell}(x)\ge I(x)$ pointwise by Remark~\ref{r:Glim}-(b), thus the conclusion.

\smallskip
\noindent (H2) $\Rightarrow$ (P2). Consider the sequence $(\nu_n)$ in $\mathcal{P}(X)$ defined as
\begin{equation*}
\nu_n(dx):= \frac{\exp(-a_n\, \varphi_n(x))}
         {\mu_n\big(\exp(-a_n\,\varphi_n)\big)} \mu_n(dx)
\end{equation*}
 By \eqref{e:Hdef}
\begin{equation*}
  \tfrac{1}{a_n} \log \mu_n\big(\exp(-a_n\, \varphi_n)\big)
 = -\nu_n (\varphi_n)- H_n^{\mb a}(\nu_n)
\end{equation*}

By \eqref{e:kkkii}, $(\nu_n)$ is tight and thus precompact in $\mathcal{P}(X)$. Let $\nu$ be an arbitrary limit point of $(\nu_n)$. Taking the
limsup in $n$, using Proposition~\ref{p:conv} and (H2)
\begin{equation*}
\begin{split}
&
\varlimsup_n \tfrac{1}{a_n} \log \mu_n\big(\exp(-a_n\, \varphi_n)\big)
 \le -\varliminf_n \nu_n\big(\varphi_n \big)
- \varliminf_n H_n^{\mb a}(\nu_n)
\\
& \qquad \qquad  \le  - \nu(\Gliminf_n \varphi_n) - \nu(I)
\le \sup_{x \in X} \big\{-(\Gliminf_n \varphi_n)(x) - I(x) \big\}
\end{split}
\end{equation*}

\smallskip
\noindent (P2) $\Rightarrow$ (P1).
Let $K$ be a compact in $X$, and for $M>0$ let $\varphi_n \equiv \varphi=
M \un {K^c}$. $(\varphi_n)$ satisfies \eqref{e:kkkii}. Moreover  $\varphi=\Glim_n \varphi_n$ since it is lower semicontinuous. Therefore assuming (P2)
\begin{equation*}
  \varlimsup_n  \tfrac{1}{a_n} \log \mu_n(K) \le 
  \varlimsup_n \tfrac{1}{a_n} \log \mu_n\big(\exp(-a_n\,M\, \un {K^c})\big)
  \le \sup_{x \in X} \{-M\, \un {K^c}(x)-I(x) \}
\end{equation*}
Letting $M \to +\infty$, (P1) follows.

\smallskip
\noindent (P1) $\Rightarrow$ (L). Fix $K\in \mf K(X)$ and let $(K_n)$ be a sequence converging to $K$ in $\mf K(X)$. Define $Q_m:=K \bigcup \cup_{n \ge m} K_n$. Then $Q_m$ is compact for all $m$, and by (P1)
 \begin{equation*}
\varliminf_n -\tfrac{1}{a_n} \log \mu_n(K_n) \ge \lim_m \varliminf_n 
-\tfrac{1}{a_n} \log \mu_n(Q_m)
 \ge \lim_m \inf_{x \in Q_m} I(x) = \inf_{x\in K} I(x)
 \end{equation*}
 where we used the lower semicontinuity of $I$ in the last inequality.
 
\smallskip
\noindent (L) $\Rightarrow$ (P1). The weak upper bound is nothing but the $\Gamma$-liminf inequality for $l^{\mb a}_n$ along a constant sequence $K_n\equiv K$.

The implications (P3) $\Rightarrow$ (P1), (P4)
  $\Rightarrow$ (P2), and \{Theorem~\ref{t:tightld}-(P), (P1)\} $\Rightarrow$ (P3)
are trivial. On the other hand the
implication \{Theorem~\ref{t:tightld}-(P),\,(P2)\} $\Rightarrow$ (P4) follows from a standard cut-off argument.
\end{proof}

\section{Applications to Large Deviations}
\label{s:4}
In this section a few consequences of the results of section~\ref{s:3} are discussed.

The following proposition gives an explicit representation of the optimal upper and low bound rate functions, see also \cite[Chapter~4.1]{DZ}, which will come useful in the following.
\begin{proposition}[Existence of Large Deviations]
\label{p:exrepr}
There exist $\upbar{I}^{\mb a}$ and $\downbar{I}^{\mb a}$ which are respectively the minimal and maximal lower semicontinuous functionals for which the weak lower bound and the upper bound
 hold respectively. A weak LDP holds for
$(\mu_n)$ with speed $(a_n)$ iff $\upbar{I}^{\mb a}=\downbar{I}^{\mb a}$.
The following representations of $\upbar{I}^{\mb a}$ and $\downbar{I}^{\mb a}$ hold
\begin{equation}
\label{e:optlow}
\upbar{I}^{\mb a}(x)
  =(\Glimsup_n H^{\mb a}_n)(\delta_x)
  =\lim_{\delta \downarrow 0} \varlimsup_n \tfrac{-1}{a_n} \log \mu_n(B_\delta(x))
  = \sup_{(V_n)} \,(\Glimsup_n V_n)(x)
\end{equation}

\begin{equation}
\label{e:optup}
\downbar{I}^{\mb a}(x)
 =(\Gliminf_n H^{\mb a}_n)(\delta_x)
 =\lim_{\delta \downarrow 0} \varliminf_n \tfrac{-1}{a_n} \log \mu_n(B_\delta(x))
 = \sup_{(V_n)}\, (\Gliminf_n V_n)(x)
\end{equation}
where the supremums are carried over all the sequences $(V_n)$ such
that $V_n\colon X \to \bb [-\infty,+\infty]$ is measurable (or
equivalently continuous and bounded) and $\mu_n (e^{a_n V_n}) \le 1$
(or equivalently $\mu_n (e^{a_n V_n}) = 1$ or equivalently $\varlimsup_n a_n^{-1} \log \mu_n(e^{a_n V_n})\le 0$).
\end{proposition}
\begin{proof}
The existence of the optimal rate functions $\upbar I^{\mb a}$ and $\downbar I^{\mb a}$, and the first representation formula above follows from the equivalences (P1) $\Leftrightarrow$ (H1) in Theorems~\ref{t:convldlow}-\ref{t:convldup}. By \eqref{e:Hcond}, it is easy to see that, given $x\in X$, the sequence $\mu_n^{B_\delta(x)}$ is an optimal recovery sequence for $\delta_x$ in the $\Gamma$-limit of $H^{\mb a}_n$, provided $\delta \downarrow 0$ after $n\to +\infty$. The second equalities in \eqref{e:optlow}-\eqref{e:optup} then follow again by (P1) $\Leftrightarrow$ (H1) in Theorem~\ref{t:convldlow}-\ref{t:convldup}. The third equalities in \eqref{e:optlow}-\eqref{e:optup} follow in the same fashion, if one remarks that the supremum in the rightest hand side is attained on the family of sequences $(V_n)$ of the form
\begin{equation*}
V_n(y):=
\begin{cases}
-\tfrac{1}{a_n} \log \mu_n(B_\delta(x)) & \text{if $y\in B_\delta(x)$}
\\
-\infty & \text{if $y \not \in B_\delta(x)$}
\end{cases}
\end{equation*}
as $\delta$ runs in $[0,1[$.
\end{proof}

The following corollaries follow easily from Proposition~\ref{p:exrepr}.
\begin{corollary}[Improving the bounds]
\label{c:bounds}
Let $\mf A$ be a set of indexes.

  Assume that for each $\alpha \in \mf A$, $\mu_n$ satisfies a LD lower bound with speed $(a_n)$ and rate $\upbar{I}_\alpha$. Then $(\mu_n)$ satisfies a LD lower bound with speed $(a_n)$ and rate equal to the lower semicontinuous envelope of $x \mapsto \inf_{\alpha \in \mf A} \upbar{I}_\alpha(x)$.

  Assume that for each $\alpha \in \mf A$, $\mu_n$ satisfies a weak
  LD upper bound with speed $(a_n)$ and lower semicontinuous rate
  $\downbar{I}_\alpha$. Then $(\mu_n)$ satisfies a weak LD upper bound with speed $(a_n)$ and rate $x  \mapsto \sup_{\alpha \in \mf A} \downbar{I}_\alpha(x)$.
\end{corollary}

\begin{corollary}[Large Deviations for double indexed sequences]
  Let $(\mu_{n,m})_{n,m}$ be a double-indexed sequence, directed by
  $(n',m')\ge (n,m)$ if $n' > n$ or $n'=n$ and $m'\ge m$.  For
  each fixed $m$ let $\upbar{I}^{\mb a}_m$ and $\downbar{I}^{\mb a}_m$ be the
  optimal lower and weak upper bound rate functionals for
  $(\mu_{n,m})_{n}$.
  
 Then $\upbar{I}^{\mb a}=\Glimsup_m \upbar{I}^{\mb a}_m$ and
 $\downbar{I}^{\mb a}=\Gliminf_m \downbar{I}^{\mb a}_m$, where $\upbar{I}^{\mb a}$ and $\downbar{I}^{\mb a}$ are defined as in \eqref{e:optlow}-\eqref{e:optup} by changing the index $n$ with $(n,m)$ (or in other words, performing the limits $n\to\infty$ and next $m\to \infty$).
\end{corollary}
%
%
%

\begin{proposition}[General contraction principle]
\label{p:contr}
  Let $X$, $Y$ be two Polish spaces, let $(\mu_n)$ be a sequence in
  $\mathcal{P}(X)$ and for $n\in \bb N$ let $\theta_n,\,\theta \colon X \to Y$ be measurable maps. Assume that $\theta_n \to \theta$ uniformly on compact sets. Define $\gamma_n = \mu_n \circ
  \theta_n^{-1} \in \mc P(Y)$. Then
  \begin{itemize}
 \item[(i)] If $(\mu_n)$ satisfies a LD lower bound with speed $(a_n)$ and lower semicontinuous rate $\upbar I \colon X \to [0,+\infty]$, then $(\gamma_n)$ satisfies a LD lower bound with the same speed and rate $\upbar J\colon Y \to [0,+\infty]$
  \begin{equation*}
  \begin{split}
\upbar J(y) & := \inf_{x \in \downbar \Lambda^y} \upbar I(x)
  \\
 \downbar  \Lambda^y
  & := \lim_{\delta \downarrow 0} \mathrm{Interior}\big(\theta^{-1}(B_{\delta}(y))\big)
  \end{split}
  \end{equation*}

  \item[(ii)] If $(\mu_n)$ is exponentially tight and satisfies a LD upper bound with speed $(a_n)$ and  lower semicontinuous rate $\downbar I \colon X \to [0,+\infty]$, then $(\gamma_n)$ satisfies a LD weak upper bound with the same speed and rate $\downbar J\colon Y \to [0,+\infty]$
  \begin{equation*}
  \begin{split}
 \downbar J(y) & := \inf_{x \in \upbar \Lambda^y} \downbar I(x)
  \\
\upbar  \Lambda^y & := \lim_{\delta \downarrow 0} \mathrm{Closure}\big(\theta^{-1}(B_{\delta}(y))\big)
  \end{split}
  \end{equation*}
  \end{itemize}
Note in particular that $\upbar \Lambda^y \supset \theta^{-1}(y) \supset \downbar \Lambda^y$, with equality holding if $\theta$ is continuous (recovering the standard contraction principle).
\end{proposition}
The proof requires a similar statement concerning $\Gamma$-convergence (contraction principles are surprisingly missing from the $\Gamma$-convergence literature).
\begin{lemma}
\label{l:contrgamma}
Let $X$, $Y$ be two Polish spaces, and let $(I_n)$ be a sequence of lower semicontinuous functions on $X$. Let $\theta_n$, $\theta$, $\upbar J$ and $\downbar J$ be as in Proposition~\ref{p:contr}. Define $J_n \colon Y \to [0,+\infty]$ as $J_n(y)=\inf_{x \in \theta_n^{-1}(y)} I_n(x)$. 
Then $\Glimsup_n J_n \le \upbar J$ and, if $(I_n)$ is equicoercive, $\Gliminf_n J_n \ge \downbar J$.
\end{lemma}
\begin{proof} Fix $y\in Y$.

\smallskip
\noindent
 \textit{$\Gamma$-limsup inequality.}
If $\downbar \Lambda^y=\emptyset$ there is nothing to prove. Otherwise, take $\eps>0$ and $x_\eps \in \downbar \Lambda^y$ such that $\upbar J(y) \ge I(x_\eps)-\eps$. Then there is a sequence $(x_{n,\eps})$ converging to $x_\eps$ in $X$ such that $\varlimsup_n I_n(x_{n,\eps}) \le I(x_\eps)$. Since $x_\eps \in \downbar \Lambda^y$, for all $\delta>0$ and $n\ge n_\delta$ large enough, $x_{n,\eps} \in \mathrm{Interior}\big(\theta^{-1}(B_\delta(x))\big)$. So that $\theta(x_{n,\eps})\to y$, and setting $y_{n,\eps} = \theta_n(x_{n,\eps})$ one has $\lim_n y_{n,\eps} = y$. On the other hand, $\varlimsup_n J_n(y_{n,\eps}) \le \varlimsup_n I(x_{n,\eps}) \le \upbar J(y)+\eps$. Thus there is a subsequence $(\eps_n)$ such that $\varliminf_n J_n(y_n) \le \upbar J_n(y)+\eps$ with $y_n = y_{n,\eps_n}$.

\smallskip
\noindent
 \textit{$\Gamma$-liminf inequality.}
 Let $(y_n)$ be a sequence converging to $y$. Up to passing to a subsequence (still label $n$ here), we can assume $\sup_n J_n(y_n)<+\infty$, the inequality being trivial otherwise. In particular, $\theta_n^{-1}(y_n) \neq \emptyset$. For $\eps>0$, let $x_{n,\eps} \in \theta_n^{-1}(y_n)$ be such that $I_n(x_{n,\eps})\le J_n(y_n)+\eps$. Since $I_n$ is equicoercive and $J_n(y_n)$ uniformly bounded, $(x_{n,\eps})$ is precompact. It is easy to check that any limit point of $(x_{n,\eps})$ is in $\upbar \Lambda^y$ (which is nonempty under the above assumptions). In particular, by the $\Gamma$-liminf inequality for $(I_n)$
\begin{equation*}
\varliminf_n J_n(y_n) \ge \varliminf_n I_n(x_{n,\eps})-\eps \ge \downbar J(y)-\eps
\end{equation*}
and we get the statement since $(y_n)$ and $\eps>0$ where arbitrary.
\end{proof}
\begin{proof}[Proof of Theorem~\ref{p:contr}]
Let $\vartheta \colon \mc P(X) \to \mc P(Y)$ be defined by $\vartheta(\mu)=\mu\circ \theta^{-1}$, and let $\vartheta_n$ be defined similarly. It is easy to see that $\vartheta_n \to \vartheta$ uniformly on compact subsets of the Polish space $\mc P(X)$. \eqref{e:Hproj} implies that for $\beta \in \mc P(Y)$
\begin{equation*}
\tfrac{1}{a_n} H(\beta| \gamma_n)= \tfrac{1}{a_n}  H(\beta| \vartheta_n(\mu_n))
= \inf_{\nu \in \vartheta_n^{-1}(\beta)}\, \tfrac{1}{a_n} H(\nu|\mu_n)
\end{equation*}
Therefore, by Lemma~\ref{l:contrgamma} (applied to the Polish space $\mc P(X)$ and maps $\vartheta_n$) and the equivalence (P1)-(H1) in Theorem~\ref{t:convldlow}, $(\gamma_n)$ satisfies a LD lower bound with speed $(a_n)$ and rate
\begin{equation*}
\begin{split}
\tilde J(y):= \inf_{\nu \in \downbar \Delta^y} \int_X \nu(dx)\,\upbar I(x)
\\
 \downbar \Delta^y=\lim_{r \downarrow 0} \mathrm{Interior}\big(\vartheta^{-1}(\mc B_r(\delta_y))\big)
\end{split}
\end{equation*}
where $\mc B_r(\delta_y)$ is the open ball of radius $r>0$ centered in $\delta_y$ in $\mc P(X)$ (with respect to a fixed compatible distance on $\mc P(X)$). However, since $\delta_x \in  \downbar \Delta^y$ iff $x \in \downbar \Lambda^y$, it is easy to see that $\tilde J= \upbar J$. Namely the statement (i) holds.

In order to prove (ii), note that by Theorem~\ref{t:tightld} the sequence of functionals $\tfrac{1}{a_n} H(\cdot|\mu_n)$ is equicoercive. One can then apply the $\Gamma$-liminf statement in Lemma~\ref{l:contrgamma}, to prove (ii) following exactly the same lines as in (i).
\end{proof}

The following result appears to be new in such a generality.
\begin{theorem}[Large deviations for coupled systems]
\label{t:coupled}
 For $n\in \bb N$, $i=1,2$ let $(\Omega^i,\,\mf F^i,\,\bb P^i_n)$ be standard probability spaces, and let $(\Omega,\,\mf F,\,\bb P_n)$ be their product space. Let $X,\,Y$ be Polish spaces with compatible distances $d_X$ and $d_Y$. Assume that for each $n$ there are measurable maps $F_n \colon Y\times \Omega^1 \to X$, $G_n \colon X\times \Omega^2 \to Y$, $\xi_n \colon \Omega \to X$ and $\eta_n \colon \Omega\to Y$ such that $\bb P_n$-a.s.
\begin{equation*}
\begin{split}
& \xi_n(\omega^1,\omega^2)=F_n(\eta_n(\omega^1,\omega^2),\omega^1)
\\
& \eta_n(\omega^1,\omega^2)=G_n(\xi_n(\omega^1,\omega^2),\omega^2)
\end{split}
\end{equation*}
For fixed $x\in X$, $y\in Y$ define $f_n^y(\omega^1)=F_n(y,\omega^1)$, $g_n^x(\omega^2)=G_n(x,\omega^2)$, and let $\mu_n^y:= \bb P^1_n \circ (f_n^y)^{-1} \in \mc P(X)$ and $\nu_n^x:=\bb P^2_n \circ (g_n^x)^{-1} \in \mc P(Y)$ be the laws of $f_n^y$ and $g_n^x$ respectively. Assume that for fixed $x\in X$, $y \in Y$ there exists  a positive function $q \equiv q^{x,y} \in C_{\mathrm{b}}(\bb R^+;\bb R^+)$ with $q(0)=0$  such that
\begin{itemize}
\item[(i)] $(\mu_n^y)$ satisfies a weak LDP with speed $(a_n)$ and lower semicontinuous rate $x \mapsto K^y(x)$.
\item[(ii)] $(\nu_n^x)$ satisfies a weak LDP with speed $(a_n)$ and lower semicontinuous rate $y \mapsto J^x(y)$.

\item[(iii)] For each $\eps>0$
\begin{equation*}
\varlimsup_n
\tfrac{1}{a_n} \log \bb P_n\Big(d_X(\xi_n,x)+d_Y(\eta_n,y) \ge q\big(d_X(f_n^y,x)+ d_Y(g_n^x,y)\big)+\eps \Big)=-\infty
\end{equation*}
\item[(iv)] For each $\eps>0$
\begin{equation*}
\varlimsup_n
\tfrac{1}{a_n} \log \bb P_n\Big(d_X(f_n^y,x)+ d_Y(g_n^x,y) \ge q\big(d_X(\xi_n,x)+d_Y(\eta_n,y)\big)+\eps \Big)=-\infty
\end{equation*}

\end{itemize}
Define $I\colon X\times Y \to [0,+\infty]$ as the lower semicontinuous envelope of the map $(x,y) \mapsto K^y(x)+J^x(y)$ and let $\gamma_n := \bb P_n \circ (\xi_n,\eta_n)^{-1}$ be the law  $(\xi_n,\eta_n)$. Then $(\gamma_n)$ satisfies a weak LDP with speed $(a_n)$ and rate $I$.
\end{theorem}
In the above theorem, (iii) and (iv) are basically uniform regularity requirements on $F_n,\,G_n$. (iii) is only used in the lower bound, (iv) in the upper bound. Theorem~\ref{t:coupled} applies in the following kind of situations. Suppose we have a weak solution to the system of SDEs on $\bb R \times \bb R$
\begin{equation}
\label{e:sdes}
\begin{cases}
& \dot \xi= b_n(\xi,\eta) + \tfrac{1}{n} \dot W^1
\\
& \dot \eta=c_n(\xi,\eta)+\tfrac{1}{n} \dot W^2
\end{cases}
\end{equation}
where $W^1$ and $W^2$ are independent Brownian motions. If one knows the LD on $C([0,T];\bb R)$ of the solutions to
\begin{equation*}
\dot \zeta = b_n(\zeta,y)  + \tfrac{1}{n} \dot W^1
\end{equation*}
\begin{equation*}
\dot \zeta=c_n(x,\zeta)+\tfrac{1}{n} \dot W^2
\end{equation*}
for fixed $x,\,y\in C([0,T];\bb R)$, and if conditions (iii)-(iv) is satisfied (which happens under uniform Lipschitz conditions on $b_n$ and $c_n$), then one gets the LD for the law of the original coupled system \eqref{e:sdes}. While this kind of statement can be quite standard for finite-dimensional systems, Theorem~\ref{t:coupled} also applies for instance in the stochastic PDEs framework, and when considering asymptotics other than the small noise limit (e.g.\ slow-fast random dynamics).

\begin{proof}[Proof of Theorem~\ref{t:coupled}]
Fix $x\in X$, $y\in Y$ and let $q\equiv q^{x,y}$ be as in the hypothesis.

\smallskip
\noindent
\textit{Lower bound.} By (i), (ii) and Theorem~\ref{t:convldlow}-(H1) there exist sequences $(\kappa_n)$ in $\mc P(X)$, $(\lambda_n)$ in $\mc P(Y)$ such that 
\begin{equation}
\label{e:alphabetan}
\begin{split}
& \kappa_n \to \delta_x \quad \text{and} \quad \varlimsup_n \tfrac{1}{a_n} H(\kappa_n|\mu_n^y)\le K^y(x)
\\
& \lambda_n \to \delta_y \quad \quad \text{and} \quad \varlimsup_n \tfrac{1}{a_n} H(\lambda_n|\nu_n^x)\le J^x(y)
\end{split}
\end{equation}
By \eqref{e:Hproj2} there exist probabilities $(\bb Q^1_n)$, $(\bb Q^2_n)$ on $(\Omega^1,\mf F^1)$, $(\Omega^2, \mf F^2)$ respectively, such that $\kappa_n=\bb Q^1_n \circ {(f_n^y)}^{-1}$, $\lambda_n=\bb Q_n^2 \circ {(g_n^x)}^{-1}$ and 
\begin{equation}
\label{e:HQP}
H(\kappa_n|\mu_n^y)=H(\bb Q^1_n|\bb P^1_n)
\qquad \qquad
H(\lambda_n|\nu_n^x)=H(\bb Q^2_n|\bb P^2_n)
\end{equation}
Set now $\bb Q_n=\bb Q^1_n \otimes \bb Q^2_n$, and define $\beta_n \in \mc P(X\times Y)$ as the law of $(\xi_n,\,\eta_n)$ under $\bb Q_n$, $\beta_n:= \bb Q_n \circ (\xi_n,\eta_n)^{-1}$. Then patching \eqref{e:alphabetan} and \eqref{e:HQP} together
\begin{equation}
\label{e:Hbeta}
\begin{split}
 \varlimsup_n \tfrac{1}{a_n} H(\beta_n|\gamma_n) & \le  \varlimsup_n \tfrac{1}{a_n} H(\bb Q_n|\bb P_n)  =
\varlimsup_n \tfrac{1}{a_n} H(\bb Q_n^1|\bb P_n^1)+ \tfrac{1}{a_n} H(\bb Q_n^1|\bb P_n^1) 
\\
 & \le \varlimsup_n \tfrac{1}{a_n} H(\kappa_n|\mu_n^y)
+ \varlimsup_n \tfrac{1}{a_n} H(\lambda_n|\nu_n^x)
\le K^y(x)+J^x(y)
\end{split}
\end{equation}
In particular, if $ K^y(x)+J^x(y)<+\infty$, $\tfrac{1}{a_n} H(\bb Q_n|\bb P_n)$ is uniformly bounded. Thus by (iii) and \eqref{e:H4}, for each $\eps>0$
\begin{equation}
\label{e:Qeq3}
\varlimsup_n
\bb Q_n\Big(d_X(\xi_n,x)+d_Y(\eta_n,y) \ge q\big(d_X(f_n^y,x)+ d_Y(g_n^x,y)\big)+\eps \Big)=0
\end{equation}
By \eqref{e:alphabetan} and $q(0)=0$, for all $\eps'>0$
\begin{equation}
\label{e:Qeq6}
\varlimsup_n \bb Q_n\big(q(d_X(f_n^y,x)+ d_Y(g_n^x,y))>\eps' \big)=0
\end{equation}
\eqref{e:Qeq3} and \eqref{e:Qeq6} yield $\beta_n \to \delta_{(x,y)}$ in $\mc P(X\times Y)$. Inequality \eqref{e:Hbeta} and Theorem~\ref{t:convldlow}-(H1) imply that the lower bound holds with rate $ K^y(x)+J^x(y)$, and by Corollary~\ref{c:bounds}, it holds with its lower semicontinuous envelope $I$.

\smallskip
\noindent
\textit{Upper bound.} Assume that $\beta_n \in \mc P(X\times Y)$ is such that $\beta_n \to \delta_{(x,y)}$. By Theorem~\ref{t:convldup}, we need to prove $\varliminf_n a_n^{-1}\,H(\beta_n|\gamma_n)\ge I(x,y)$. Up to passing to a subsequence, one can assume $a_n^{-1} H(\beta_n|\gamma_n)$ to be bounded uniformly in $n$, so that by the Remark~\ref{r:relentr} there exists a probability $\bb Q_n$ on $(\Omega,\mf F)$ such that $\beta_n =\bb Q_n \circ (\xi_n,\eta_n)^{-1}$ and $H(\beta_n|\gamma_n)= H(\bb Q_n|\bb P_n)$. Let $\bb Q_n^1(d\omega^1)$ and $\bb Q_n^2(d\omega^2)$ be the marginals of $\bb Q_n$ on $\Omega^1$ and $\Omega^2$ respectively. For all $\eps>0$
\begin{equation}
\label{e:Qeq7}
\begin{split}
& \bb Q_n^1(d_X(f^y_n,x)>\eps)+\bb Q_n^2(d_Y(g^x_n,y)>\eps)
\\
&\quad =\bb Q_n(d_X(f^y_n,x)>\eps)+\bb Q_n(d_Y(g^x_n,y)>\eps)
\le 2\,  \bb Q_n \big(d_X(f_n^y,x)+ d_Y(g_n^x,y)>\eps\big)
\\
& \quad \le 2\,\bb Q_n\Big(d_X(f_n^y,x)+ d_Y(g_n^x,y)> q\big(d_X(\xi_n,x)+d_Y(\eta_n,y)\big)+\eps/2 \Big)
\\
 & \qquad +2\,\bb Q_n\Big(q\big(d_X(\xi_n,x)+d_Y(\eta_n,y)\big)>\eps/2 \Big)
\end{split}
\end{equation}
The last line of \eqref{e:Qeq7} vanishes as $n\to +\infty$, since $\beta_n \to \delta_{x,y}$. 
On the other hand, by \eqref{e:H4} and hypothesis (iv) 
\begin{equation*}
\begin{split}
& \varlimsup_n\bb Q_n\Big(d_X(f_n^y,x)+ d_Y(g_n^x,y)> q\big(d_X(\xi_n,x)+d_Y(\eta_n,y)\big)+\eps/2 \Big)
\\
& \quad \le \varlimsup_n \frac{\frac{\log 2}{a_n} + \frac{1}{a_n}H(\bb Q_n|\bb P_n)}{-\frac{1}{a_n} 
 \log \bb P_n\Big(d_X(f_n^y,x)+ d_Y(g_n^x,y)> q\big(d_X(\xi_n,x)+d_Y(\eta_n,y)\big)+\eps/2  \Big)}=0
\end{split}
\end{equation*}
Thus by \eqref{e:Qeq7}, for all $\eps>0$
\begin{equation*}
\lim_n \bb Q_n^1(d_X(f^y_n,x)>\eps)=\lim_n\bb Q_n^2(d_Y(g^x_n,y)>\eps)=0
\end{equation*}
and, letting $\kappa_n=\bb Q^1_n \circ {(f_n^y)}^{-1} \in \mc P(X)$, $\lambda_n=\bb Q_n^2 \circ {(g_n^x)}^{-1} \in \mc P(Y)$, we gather
\begin{equation}
\label{e:convQ12}
\lim_n \kappa_n = \delta_x \quad \text{in $\mc P(X)$} \qquad \quad 
\lim_n \lambda_n = \delta_y \quad \text{in $\mc P(Y)$} 
\end{equation}

Disintegrate now $\bb Q_n$ as $\bb Q_n(d\omega^1,d\omega^2)= \bb Q_n^1(d\omega^1) {\bf Q}_n(\omega^1;d \omega^2)$. Then by explicit calculations and Jensen inequality
\begin{equation}
\label{e:Hboundeasy}
\begin{split}
H(\beta_n|\gamma_n)&  = H(\bb Q_n|\bb P_n)= H(\bb Q_n^1|\bb P_n^1)+ \int_{\Omega^1} \bb Q_n(d\omega^1)\,H({\bf Q}_n(\omega^1;\cdot)|\bb P_n^2) 
\\
& \ge H(\bb Q_n^1|\bb P_n^1)+
H\Big(\int_{\Omega^1} \bb Q_n(d\omega^1)\,{\bf Q}_n(\omega^1;\cdot) \Big|\bb P_n^2\Big) 
\\
& = H(\bb Q_n^1|\bb P_n^1)+ H(\bb Q_n^2|\bb P_n^2)+
\ge H(\kappa_n^y|\mu_n^y) +H(\lambda_n^y|\nu_n^x)
\end{split}
\end{equation}
By \eqref{e:convQ12}, hypotheses (i), (ii) and Theorem~\ref{t:convldup}-(H1)
\begin{equation*}
\begin{split}
\varliminf_n \tfrac{1}{a_n}H(\beta_n|\gamma_n)&  \ge 
  \varliminf_n \tfrac{1}{a_n} H(\kappa_n^y|\mu_n^y)
   +\varliminf_n \tfrac{1}{a_n} H(\lambda_n^y|\nu_n^x)
 \\
  & \ge K^y(x)+J^x(y) \ge I(x,y)
\end{split}
\end{equation*}
concluding the proof.
\end{proof}


\section{LD and $\Gamma$-convergence topology}
\label{s:5}
We say that the speed $\mb a$ is trivial for the LD of $(\mu_n)$ if the functionals $\upbar{I}^{\mb a}$, $\downbar{I}^{\mb a}$ only take the values $0$ and $+\infty$. Assume, for the sake of simplicity, that $(\mu_n)$ converges to $\mu$ in $\mc P(X)$, and note
\begin{equation}
\label{e:suppinc}
\begin{split}
\mathrm{Support}(\mu)
 \subset
\mathrm{Closure}\big(
\varliminf_n \mathrm{Support}(\mu_n)\big)
\end{split}
\end{equation}
If the inclusion \eqref{e:suppinc} is actually an equality, which means that the measures $\mu_n$ do not feature any concentration phenomena in the limit $n\to \infty$, then it is easy to check that
\begin{equation*}
\upbar{I}^{\mb a}(x)=\downbar{I}^{\mb a}(x)=
\begin{cases}
0 & \text{if $x \in \mathrm{Support}(\mu)$}
\\
+\infty & \text{otherwise}
\end{cases}
\end{equation*}
regardless of the speed $(a_n)$. That is, the LD of $(\mu_n)$ are trivial. On the other hand, if the inclusion \eqref{e:suppinc} is strict, one can prove that there exists a non-trivial speed $(a_n)$. This remark suggests that, when considering LD as a notion of convergence on the space of couples $(a,\mu) \in \bb R^+ \times \mathcal{P}(X)$, one should identify the singular measures, since no speed $(a_n)$ can catch the concentration speed of the support of Dirac masses. More precisely, recall \eqref{e:lamu}, and define the equivalence relation on $\bb R^+ \times \mathcal{P}(X)$
 \begin{equation*}
(a,\mu) \sim (a',\mu')  \quad \Leftrightarrow \quad \{ (a,\mu)=(a',\mu') \text{ or } \exists x\in X\,:\:\mu=\mu'=\delta_x\} \quad \Leftrightarrow \quad l_{a,\mu}=l_{a',\mu'}
 \end{equation*}
and let
\begin{equation*}
\begin{split}
\mc U(X) & :=  \bb R^+ \times \mathcal{P}(X) / \sim
\\ 
\mc V(X) &:= \big\{ I\colon X \to [0,+\infty],\,\text{$I$ is lower semicontinuous}\big \}
\\
\mc W(X) & := \mc U(X) \cup \mc V(X)
\\
\mc L(X) & := \big\{l\colon\mf K(X) \to [0,+\infty],\,\text{$l$ is lower semicontinuous, $l(K) \le l(K')$ if $K \supset K'$}  \big\}
\end{split}
\end{equation*}
We want to look at LD as a notion of convergence in $\mc W(X)$. To our aim, $\mc L(X)$ is naturally equipped with the topology of $\Gamma$-convergence \cite[Chapter 10]{DM} on $\mf K(X)$. One can prove that $\mc L(X)$ is a $T_1$, supercompact space (an easy extension of \cite[Theorem~10.6]{DM}). The maps \eqref{e:lamu}, \eqref{e:lI} define an injection $\mc W(X) \hookrightarrow \mc L(X)$, and we equip $\mc W(X)$ with the induced topology. We say that a subset $\mc W(X)$ is equicoercive, if its homeomorphic image in $\mc L(X)$ is equicoercive.
Note in particular that a sequence $(a_n,\mu_n)$ is equicoercive  iff $(\mu_n)$ is exponentially tight with speed $(a_n)$.

The following theorem is a consequence of the equivalence between the (L) and (P) statements in Proposition~\ref{p:conv} and Theorems~\ref{t:convldlow}-\ref{t:convldup}, and the metrizability properties of the topology of $\Gamma$-convergence for equicoercive subsets \cite[Theorem~10.22]{DM}.
\begin{theorem}
\label{t:topol}
Let $(w_n)$ be a sequence converging to $w$ in $\mc W(X)$. Then
\begin{itemize}
\item[(i)] Up to $\sim$ identification, if $w_n=(a_n,\mu_n)$ and $(a_n)$ is bounded, then $a_n \to a$ and $\mu_n \to \mu$ and $w=(a,\mu) \in \mc U(X)$.
\item[(ii)] If $w_n=(a_n,\mu_n)$ and $a_n \to +\infty$, then $w=I \in \mc V(X)$ and $\mu_n$ satisfies a weak LDP with speed $(a_n)$ and rate $I$.
\item[(iii)] If $w_n=I_n$, then $w=I$ and $I_n$ $\Gamma$-converges to $I$.
\end{itemize}
The relative topology induced by $\mc W(X)$ on an equicoercive subset $\mc E$ is metrizable.
\end{theorem}
Roughly speaking, the previous theorem states that the topology induced by $\mc E$ on measures is the usual topology of narrow convergence, the topology it induces on functionals is the topology of $\Gamma$-convergence. However, while the space of functionals on $X$ is compact under this topology, it can happen that measures converge to functionals, and this is the case iff a LDP holds. It is worth to remark that up to identification $\mc W(X)$ can be regarded as a subset of the set $\mc Q(X)$ introduced in \cite[Chapter 4.7]{DZ}, and while the topology of $\mc L(X)$ does not induce the topology on $\mc Q(X)$ therein considered,  the two topologies coincide on $\mc W(X)$.

%

\section{Second order Sanov Theorem}
\label{s:6}
In this section we give a simple application of the results in section~\ref{s:3}. Sanov Theorem states that, if the random variables $(x_i)_{i\in \bb N}$ are i.i.d.\ with law $\mu \in \mc P(X)$, then the law of their empirical measure satisfies a LDP with speed $(n)$ and rate $H(\cdot|\mu)$. The result also holds if the law $\mu_n$ of the random variables depends on $n$, provided $\mu_n \to \mu$. However, if $(\mu_n)$ concentrates in the sense of \eqref{e:suppinc}, the LD  can admit a nontrivial "second order" expansion. To fix the idea, suppose that $\mu_n \to \mu =\delta_x$ for some $x\in X$. Then the functional $H(\cdot|\delta_x)$ is trivial (in the sense of section~\ref{s:5}), and the speed $(n)$ is not the interesting one. More in general, several non-trivial LDPs may hold, as shown in the following theorem (this result already appeared in the literature in the context of diffusion processes \cite{DG}, and closer to the framework of this paper in \cite{Le}).
\begin{theorem}
\label{t:sanov}
Let $(\mu_n)$ be a sequence converging to $\mu$ in $\mc P(X)$, and define the empirical measure $\pi_n \colon X^n \to \mc P(X)$ as
\begin{equation*}
\pi_n:= \tfrac{1}{n} \sum_{i=1}^n \delta_{x_i}
\end{equation*}
Then the law $\bb P_n:=\mu_n^{\otimes n} \circ \pi_n^{-1}$ of $\pi_n$ under the product measure $\mu_n^{\otimes n}$ satisfies a LDP on $\mc P(X)$ with speed $(n)$ and rate $H(\cdot|\mu)$.

Assume furthermore that $(\mu_n)$ satisfies a LDP on $\mc P(X)$ with speed $(a_n)$ and lower semicontinuous, coercive rate $I \colon \mc P(X) \to [0,+\infty]$. Then the law of $\pi_n$ satisfies a LDP on $\mc P(X)$ with speed $(n\,a_n)$ and lower semicontinuous, coercive rate $\mc I \colon \mc P(X) \to [0,+\infty]$
\begin{equation*}
\mc I(\nu):= \int_X \nu(dx)\,I(x)
\end{equation*}
\end{theorem}
\begin{proof}
Fix $\nu \in \mc P(X)$.

\smallskip
\noindent
\textit{Lower bound with speed $(n)$.} By Proposition~\ref{p:conv}-(H), there exists $\nu_n \to \nu$ such that $\varlimsup_n H(\nu_n|\mu_n) \le H(\nu|\mu_n)$. Take $\bb Q_n:= \nu_n^{\otimes n} \circ \pi_n^{-1}$. Then $\bb Q_n \to \delta_{\nu}$ and 
\begin{equation}
\label{e:Hextgamma}
\tfrac{1}{n} H(\bb Q_n|\bb P_n) =\tfrac{1}{n} H(\nu_n^{\otimes n} \circ \pi_n^{-1}|\mu_n^{\otimes n} \circ \pi_n^{-1}) \le \tfrac{1}{n}  H(\nu_n^{\otimes n}|\mu_n^{\otimes n})= H(\nu_n|\mu_n)
\end{equation}
so that we conclude by Theorem~\ref{t:convldlow}-(H1).

\smallskip
\noindent
\textit{Lower bound with speed $(n\,a_n)$.} By Theorem~\ref{t:convldlow}-(H2), there exists there exists $\nu_n \to \nu$ such that $\varlimsup_n \tfrac{1}{a_n} H(\nu_n|\mu_n) \le \mc I(\nu)$. Take $\bb Q_n:= \nu_n^{\otimes n} \circ \pi_n^{-1}$. By the same calculation as in \eqref{e:Hextgamma} and Theorem~\ref{t:convldlow}-(H1) we conclude.

\smallskip
\noindent
\textit{Weak upper bound with speed $(n)$.} Let $(\bb Q_n)$ be a sequence in $\mc P(\mc P(X))$ such that $\bb Q_n \to \delta_\nu$. We want to prove $\varliminf_n H(\bb Q_n|\bb P_n) \ge H(\nu|\mu)$. One can assume $\bb Q_n= \gamma_n \circ \pi_n^{-1}$ for some $\gamma_n \in \mc P(X^n)$ the relative entropy being infinite otherwise, see \eqref{e:Hproj}. Since $\pi_n({\bf x})=\pi_n({\bf x'})$ iff ${\bf x'}$ is obtained from ${\bf x}$ by an index permutation, $\gamma_n$ can be assumed invariant under index permutation as well, see \eqref{e:Hproj2}, to obtain
\begin{equation*}
H(\bb Q_n |\bb P_n)= H(\gamma_n \circ \pi_n^{-1}|\mu_n^{\otimes n}\circ \pi_n^{-1})=H(\gamma_n|\mu_n^{\otimes n})
\end{equation*}
Let $\nu_n \in \mc P(X)$ be the one-dimensional marginal of $\gamma_n$. By the explicit representation \eqref{e:Hdef} of the relative entropy and its convexity, reasoning as in \eqref{e:Hboundeasy}
\begin{equation*}
H(\bb Q_n |\bb P_n)= H(\gamma_n|\mu_n^{\otimes n}) \ge H(\nu_n^{\otimes n}|\mu_n^{\otimes n})= n \,H(\nu_n|\mu_n)
\end{equation*}
On the other hand $\bb Q_n \to \delta_\nu$ implies $\nu_n \to \nu$, so that by Proposition~\ref{p:conv}-(H)
\begin{equation*}
\varliminf_n \tfrac{1}{n} H(\bb Q_n|\bb P_n) \ge \varliminf_n H(\nu_n|\mu_n) \ge H(\nu|\mu)
\end{equation*}

\smallskip
\noindent
\textit{Weak upper bound with speed $(n\,a_n)$.} Following the same strategy of the bound with speed $(n)$, we obtain
\begin{equation*}
\varliminf_n \tfrac{1}{n\,a_n} H(\bb Q_n|\bb P_n) \ge \varliminf_n \tfrac{1}{a_n} H(\nu_n|\mu_n) \ge \int \nu(dx)\,I(x)
\end{equation*}
where in the last inequality we used Theorem~\ref{t:convldup}-(H2). We conclude by applying Theorem~\ref{t:convldup}-(H1) to $(\bb Q_n)$.

\smallskip
\textit{Exponential tightness.} With the same notation of the upper bound proofs, if $\tfrac{1}{n} H(\bb Q_n|\bb P_n)$ is uniformly bounded then $H(\nu_n|\mu_n)$ also is. Since $\mu_n$ is tight, Proposition~\ref{p:tight} implies that $\nu_n$ is tight and thus $\bb Q_n$ is tight as well.
By Theorem~\ref{t:tightld} we conclude that $(\bb P_n)$ is exponentially tight with speed $(n)$. Since $(\mu_n)$ satisfies a LDP with speed $(n\,a_n)$ and $I$ is coercive, $(\mu_n)$ is exponentially tight with this speed \cite[Ex.\ 4.1.10]{DZ}, and the same proof yields that $(\bb P_n)$ is exponentially tight with speed $(n\,a_n)$.
\end{proof}

An example of application of Theorem~\ref{t:sanov}, is the extension of well known results about LD for the empirical measure of independent random walks or diffusion processes. Let $(X_i^n)_{i\in \bb N}$ be a family of stochastic process $X_i^n \in D([0,+\infty[; \bb R^d)$, all starting at ${\bf 0}$ (the case of different initial conditions for each $X_i^n$ could also be fitted in this framework, but we keep the notation simple). If $X_i^n$ converges in law to some limit process $X$ (e.g.\ $X_i^n$ is a parabolically-rescaled symmetric random walk converging to a brownian motion or a Levy process in case of heavy tails), then the LD of the empirical measure $\pi_n$ happen with speed $(n)$, for instance the results in \cite{KO} can be recovered from the first part of Theorem~\ref{t:sanov} by a contraction principle. However, if $X_i^n$ converges to a deterministic trajectory (e.g.\ $X_i^n$ is a hyperbolically-rescaled asymmetric random walk, converging to a uniform motion), then the LD  happen with a faster speed. For instance one can recover the results in \cite{KL} by the second part of Theorem~\ref{t:sanov} by a contraction principle.

A most interesting open problem related to the above framework is the analysis of the LD for the empirical measure of a totally asymmetric simple exclusion process on $\bb Z$ (TASEP). Indeed, the law of the path of a particle $X_i$ is independent of the law of the other particles conditionally to the path of the particle at its right (provided the TASEP moves right). TASEP would therefore fit in the framework of Theorem~\ref{t:sanov}, except that $(x_i)$ is now a Markov chain (not an i.i.d.\ sequence), and $\mu_n$ is replaced by a jump kernel $\mu_n(x,dy)$. In this case, one still expects the presence of multiple non-trivial speeds for the LD in the same fashion of Theorem~\ref{t:sanov}; however the Markov equivalent of Theorem~\ref{t:sanov} features a richer description, and it is still subject of investigation.

\medskip
\textit{Acknowledgment:} The author has discussed the ideas in this paper with several people, who motivated and helped him to provide a systematic treatment of the subject. In this respect, I especially acknowledge G.Bellettini, L.Bertini, R.Cerf, L.Zambotti. I am also grateful to F.Cagnetti for helpful discussions about \cite[Chapter 10]{DM}. This work has been supported by the PRIN 20155PAWZB \emph{Large Scale Random Structures}. I acknowledge the support of Dipartimento di Matematica, Sapienza Università di Roma.


\begin{thebibliography}{63}
\bibitem{Bi} Billingsley P., {\sl Convergence of Probability measures}, New York, John Wiley and Sons, 1999, 2nd Edition.


\bibitem{DM} Dal Maso G., {\sl An introduction to $\Gamma$-convergence.} Progress in Nonlinear Differential Equations and their Applications \textbf{8}, Birkhauser, 1993.


\bibitem{DE} Dupuis, Ellis, {\sl A weak convergence approach to the theory of large deviations}, Wiley-Interscience 1997.

\bibitem{DG} Dawson D.A., G\"artner, J.,
{\sl Multilevel large deviations and interacting diffusions},
Probab.\ Theory Related Fields 98-4, 423-487, 1994. 

\bibitem{DZ} Dembo A., Zeitouni O., {\sl Large Deviations Techniques and Applications}, Jones and Bartlett Publishers, 1993.

\bibitem{J} Jensen L., {\sl Large deviations of the asymmetric simple exclusion process in one dimension}, PhD Thesis, New York University, 2000.

\bibitem{KL} Kipnis C., L\'eonard C., {\sl Grandes D\'eviations pour un syst\`eme hydrodynamique asym\'etrique de particules ind\'ependentes}, Ann.\ Inst.\ H.Poincar\'e, Probabilit\'es, \textbf{31}, 233-248, 1995.

\bibitem{KO} Kipnis C., Olla S., {\sl Large deviations from the hydrodynamical limit for a system of independent brownian particles}, Stochastics Stochastics Rep., \textbf{33}, 17--25, 1990.

\bibitem{Le} L\'eonard C.,
{\sl From the Schr\"odinger problem to the Monge-Kantorovich problem}, J.\ Funct.\ Anal.\ 262-4, 1879-1920, 2012.

\bibitem{Ma} Mariani M., {\sl Large deviations principles for stochastic scalar
conservation laws}, Probability Theory and Related Fields, Vol.\ 147 (2010).

\bibitem{RW} Rockafellar R.T., Wets T., {\sl Variational Analysis}, Springer-Verlag 2009.

\bibitem{St} Stroock D.W., {\sl Probability Theory. An analytic view}, Cambridge University Press, 1994.

\end{thebibliography}
\end{document}